\documentclass[12pt,thmsa]{amsart}
\usepackage[dvips]{graphics}
\input{amssym.def}
\input{amssym.tex}

\def\Cal{\mathcal}

\def\D{{\Cal D}}

\def\bbr{{\Bbb R}}

\def\bbh{{\Bbb H}}

\def\bbc{{\Bbb C}}

\def\bbe{{\Bbb E}}

\def\dist{{\hbox{\rm dist}}}

\def\supp{{\hbox{\rm supp}}}
\def\sgn{{\hbox{\rm sgn}}}

\def\const{{\hbox{\rm const}}}

\def\rn{\bbr^n}

\def\part{\partial}
\def\intl{\int\limits}

\def\Gam{\Gamma}

\def\a{\alpha}
\def\om{\omega}

\def\Del{\Delta}
\def\del{\delta}
\def\vp{\varphi}

\def\gam{\gamma}

\def\sig{\sigma}
\def\lam{\lambda}
\def\z{\zeta}
\def\e{\varepsilon}
\def\t{\tau}

\newtheorem{theorem}{Theorem}[section]
\newtheorem{lemma}[theorem]{Lemma}

\theoremstyle{remark}
\newtheorem{remark}[theorem]{Remark}

\numberwithin{equation}{section}

\newcommand{\be}{\begin{equation}}
\newcommand{\ee}{\end{equation}}

\newcommand{\bea}{\begin{eqnarray}}
\newcommand{\eea}{\end{eqnarray}}
\newcommand{\Bea}{\begin{eqnarray*}}
\newcommand{\Eea}{\end{eqnarray*}}

\def\sideremark#1{\ifvmode\leavevmode\fi\vadjust{\vbox to0pt{\vss
 \hbox to 0pt{\hskip\hsize\hskip1em
\vbox{\hsize2cm\tiny\raggedright\pretolerance10000
 \noindent #1\hfill}\hss}\vbox to8pt{\vfil}\vss}}}%

\begin{document}

\title[Spherical Means ]
{Inversion Formulas for the Spherical Means in Constant Curvature Spaces}

\author{Yuri A. Antipov}
\address{Department of Mathematics, Louisiana State University, Baton Rouge,
LA, 70803 USA} \email{antipov@math.lsu.edu}

\author{Ricardo  Estrada}
\address{Department of Mathematics, Louisiana State University, Baton Rouge,
LA, 70803 USA} \email{restrada@math.lsu.edu}

\author{Boris Rubin}
\address{
Department of Mathematics, Louisiana State University, Baton Rouge,
LA, 70803 USA}

\email{borisr@math.lsu.edu}
\thanks{The  research of the first and the second author was supported  by the NSF grants  DMS-0707724 and PHYS-0968448, respectively.
The   third author was supported in part by the  NSF grant DMS-0556157
and the Louisiana EPSCoR program, sponsored  by NSF and the Board of
Regents Support Fund.}





\begin{abstract}   

The work develops further the theory of the following inversion problem, which plays the central role in the rapidly developing area of thermoacoustic tomography and has intimate connections with PDEs 
 and integral geometry: {\it  Reconstruct a function $f$  supported in an $n$-dimensional ball $B$, 
  if the spherical means of $f$  are known over all geodesic spheres centered on the
boundary of $B$.} 
We propose a new unified approach
based on the idea of
analytic continuation. 
 This approach gives explicit inversion formulas
not only for the Euclidean space $\bbr^n$ (as in the original set-up) 
 but also for  arbitrary constant curvature space $X$, including 
  the $n$-dimensional
sphere and
the hyperbolic space. The results  are applied  to inverse problems for a large 
class of Euler-Poisson-Darboux equations in  constant
curvature spaces of arbitrary dimension.
\end{abstract}

\maketitle

\section{Introduction}

\setcounter{equation}{0}
The paper deals with the spherical mean operator, which is also known as the spherical mean Radon transform. Importance of this transformation in analysis and geometry and 
 many of its properties (which are  still surprising!)  were indicated by many authors; see, e.g.,  
 \cite[p. 699]{CH},  \cite{Jo,  St}.  
  In recent years an interest to this object has grown  tremendously in view of a series of challenging  problems. One of them is characterization of sets of injectivity  (and non-injectivity) of this transform; see  \cite{ EKl, AQ1, AQ2, Sr}  and references therein.  Another source of  mathematical problems related to the spherical means is the 
 rapidly developing thermoacoustic 
 tomography (TAT), the revolutionary role of which in medical imaging was pointed out  in many publications; see \cite{AFK}-\cite{AKQ}, \cite {FHR,FPR}, \cite{FR}-\cite{FHM}, \cite{HKN,K}, \cite{KLFA}-\cite{KK}, \cite{ Ku}, \cite{ OK}-\cite{PS2}, \cite{ Wa,XW}.  The present investigation belongs to this area.

{\bf Setting of the problem and motivation.}  Let $f$ be an infinitely differentiable  function with compact support in the open ball $B=\{x \in \rn: |x|<R\}$;  $ \partial B$ is the boundary 
 of $B$.   
We consider the spherical mean Radon transform $Mf$ which integrates $f$ over spheres centered on  $ \partial B$:  
\be\label {mf}
(Mf)(\xi, t)=\frac{1}{\sig_{n-1}}\int_{S^{n-1}} f (\xi
-t\theta)\, d\theta,\ee 
where
$$ \xi \in \partial B, \qquad t\in
\bbr_+=(0,\infty),$$ 
$S^{n-1}$ is the unit
sphere in $\rn$ with the area $\sig_{n-1}$,
and $d\sig$ stands for the usual Lebesgue measure
on $S^{n-1}$. For the classical Radon transforms, their modifications,  and applications  see, e.g., 
 \cite{Dea, Ep, GGG, Eh, He, N}.

The general 
 problem  of reconstructing $f$ from known data $(Mf)(\xi, t)$ on the cylinder  $\partial B \times \bbr_+$ is an immediate consequence of 
 the following commonly accepted mathematical model of TAT in $\bbr^3$  (see, e.g., \cite {KK, Wa, HKN} and references therein): 
 
 {\it Given a function $c(x)$,  the speed of the ultrasound propagation in the tissue, and a function
 $g(\xi, t)$, the measured  value of the pressure at the time $t$  at the transducerÕs location $\xi\in S^2$, find a function $f(x)$, the initial pressure distribution $p(x,0)$
 (the TAT image), if 
 \be\label {mmo}   \left \{\begin{array} {ll} p_{tt} = c^2(x)\, \Del  p \quad & \mbox{\rm for all} \;  t\ge 0, \;x\in \bbr^3,\\
 p(x,0)= f(x),  \; p_t(x,0) = 0 \quad &  \mbox{\rm for all} \; x\in \bbr^3,\\p(\xi,t)=g(\xi,t) \qquad &  \mbox{\rm for all} \; \xi\in S^2 \, (\subset \bbr^3), \; t\ge 0.\\ \end{array}
\right.\ee
}
Here, $p_t $ and 
$p_{tt}$ are the first and second time derivatives, and $\Del$  is the Laplace operator with respect to the spatial variable $x$.

This problem  admits immediate generalization to arbitrary dimensions, more general Riemannian   spaces, and a broad class of  differential equations of the Euler-Poisson-Darboux  (EPD) type.  In the Euclidean case,  a thorough discussion of main inversion methods  in terms of their assumptions and computational features can be found in  \cite{KK, Wa}.  In the important particular case of constant speed $c(x)$, solution to   this problem is equivalent to reconstruction of $f$ from its spherical mean   (\ref{mf}).
 
Explicit inversion formulas for $Mf$  are of
particular  interest.    For $n$ odd, such formulas were obtained by Finch, Patch, and
 Rakesh  in \cite{FPR}.  Another derivation was suggested by Palamodov  \cite[Section 7.5]{P}; see also \cite {FR2, FR3, XW}. The corresponding formulas for  $n$ even were
 obtained by Finch,  Haltmeier, and  Rakesh in \cite{FHR}.  An   explicit inversion formula which relies on  completely different ideas 
 and covers both odd and even cases, was suggested by  Kunyansky \cite{Ku}.
 
In spite of the elegance and ingenuity,  the derivation of the existing   inversion formulas for $Mf$  is pretty involved, and basic ideas behind it remain mysterious.  In view of practical importance, it would be desirable to find an independent simple proof of known formulas and thus check their correctness. Moreover,  the prospective new method should  be applicable to more general geometric and analytic settings and thus lead to further progress.

A simple proof for $n$ odd was suggested by the third author \cite {Ru00a},  who suggested to treat $M$ as a member of a  certain analytic family of operators and applied the results to  the inverse problem of type (\ref{mmo})  for the more general  EPD  equation in the case $c(x)\equiv \const$.

In the present article we suggest a new approach, which is conceptually simple and leads to inversion formulas for $Mf$  {\it for all}  $n\ge 2$. 
As in \cite {Ru00a}, the  key idea is  analytic continuation, however, the reasoning is different. We extend our method  to the similar problem for spherical means on the $n$-dimensional sphere and the hyperbolic space, where the theory of  EPD  equations is also well-developed. This extension seems to be new and paves the way to diverse settings,  when  the relevant geodesic balls and spherical means   are considered in more general Riemannian spaces.

Regarding  generalizations to general Riemannian spaces, some comments are in order.   The corresponding wave equations and their 
EPD generalizations were studied in \cite{LP, KI1, KI2}, \cite{O1}-\cite{O3}. For example, the wave equation on the $n$-dimensional sphere $S^n$   has the form \cite{LP}
\be\label{maw}
 \del_x u = u_{\om\om} + \left(\frac{n-1}{2}\right)^2 u,  \qquad (x,\om)\in  
S^n \times  (0,\pi),\ee
where  $\del_x$ denotes   the  Beltrami-Laplace operator.
The Cauchy problem
 for the relevant EPD equation 
\be\label {papo}  \tilde\square_\a u=0, \qquad u(x,0) = f(x), \quad  u_\om (x, 0) = 0,\ee
where
\be\label{wder}
\tilde\square_\a u=\del_x u - u_{\om\om} - (n-1+2\a) \cot \om\, u_\om + \a (n-1+\a)u,
\ee
and various modifications  were discussed  in  \cite{C, CS, F, KI1, KI2, O2}.
 The problem  (\ref{papo})  for the case  $\a = 0$, corresponding to the usual 
Darboux  equation, was studied  by  Olevskii  \cite{O1} and also    by
  Kipriyanov and  Ivanov   \cite{KI1}.
Our definition of the EPD-equation  on $S^n$ differs from that in 
 \cite{KI1}  and agrees with  \cite{O2}. The particular case $\a=(1-n)/2$,  corresponding to the wave equation (\ref{maw}),   can be regarded as the spherical analogue of the  TAT model  (\ref{mmo}) with constant speed.

\smallskip

{\bf Plan of the paper and main results.}   Section 2 contains preliminaries. Here the main statement is Lemma  \ref{lem3}.
For convenience of the reader and  better treatment of the subject, we supply this lemma with  alternative proofs, which are based on different ideas and use different tools, while leading to the same result. All these are presented  in Appendix.   Section 3 contains derivation of inversion formulas for $Mf$ in the Euclidean case.  The main inversion results are presented in Theorems \ref{mku} and  \ref{Theorem n2.1};  see also modified inversion formulas (\ref{inv1m}),  (\ref{inv2m}).  
 The results of Section 3 are applied in Section 4 to  
  the Cauchy problem for the Euler-Poisson-Darboux equation    
\begin{equation}
\square_{\a}u\!\equiv\!\Delta u\!-\!u_{tt}\!-\!\frac{n\!+\!2\alpha\!-\!1}{t}\,u_{t}%
\!=\!0,\quad u(x,0)\!=\!f(x),\;  u_{t}(x,0)\!=\!0,\label{zas}%
\end{equation}%
  where $f$ is a smooth function with compact support in the ball $B$. 
Using the results of Section 3  combined with known properties of Erd\'{e}lyi-Kober fractional integrals,  we give explicit solution (Theorem \ref{Th4.1})  to the following inverse problem:   

\textit{ Given   the trace $u\left( \xi,t\right)  $
 of  the solution  of  (\ref{zas})  for all  $\left(\xi,t\right)$  on the cylindrical surface $\partial B\times\mathbb{R}_{+}$,  reconstruct $f(x)$. }

The particular case $\a=(1-n)/2$  gives explicit solution to the TAT problem (\ref{mmo})
 with constant speed $c(x)\equiv 1$.

 The spherical mean Radon transform
 on the $n$-dimensional unit sphere $S^n$  in $\bbr^{n+1}$  is studied in Section 5.  Inversion formulas 
 for this transform are given  in Theorems \ref{774}  and \ref{Theorem n2.1s2}.
  The relevant inverse problem  for the EPD
 equation  on $S^n$  is solved 
 in Section 6.   Section 7 contains derivation of   inversion formulas for the spherical mean Radon transform in the  $n$-dimensional hyperbolic space. Here the main results are given by Theorems \ref{774h} and  \ref  {Theorem n2.1h} .

 {\bf Acknowledgements.}  The third author  is  grateful to  Mark Agranovsky,  who encouraged him to study this problem,  and also  to Peter
 Kuchment  and Leonid Kunyansky for useful  discussions. Special thanks go to David Finch,  who shared with us his  knowledge of the subject.

\section{Auxiliary statements}

 {\bf Notation.} We use   abbreviation $a.c.$  to denote analytic continuation;  $\sig_{n-1}\!=\!2\pi^{n/2}/\Gam (n/2)$ is the area of the unit sphere $S^{n-1}$  in $\rn$. We write  $\;d\theta$     ($d\xi$)  for the usual Lebesgue measure
on $S^{n-1}$  (on $\partial B$,  resp.);  $[a]$ denotes the integer part of a real number $a$;   $(\cdot )^\lam_+$  means  $(\cdot )^\lam$ if the expression in parentheses is positive and zero, otherwise.

We will need the following  lemmas.

\begin{lemma} \label {lem1} Let $\vp \in C_c^\infty (\bbr)$.

\noindent {\rm (i)} If $m=0,1,2,
\ldots$, then
\be\label {lab1c}
\underset
{\a=-2m}{a.c.} \intl_\bbr \frac{|t|^{\a -1}}{\Gam ( \a/2)} \, \vp (t)\, dt=c_{m,1}\,  \vp^{(2m)} (0), \quad c_{m,1}=\frac{(-1)^m\, m!}{(2m)!}.\ee

\noindent {\rm (ii)} If $m=1,2,
\ldots$, then
\bea\label {lab1z}
\underset
{\a=1-2m}{a.c.} \intl_\bbr \frac{|t|^{\a -1}}{\Gam ( \a/2)} \, \vp (t)\, dt&=&c_{m,2} \intl_\bbr \frac{\vp^{(2m-1)} (t)}{t}\, dt\\
&=&-c_{m,2} \intl_\bbr \vp^{(2m)} (t)\,\log |t|\, dt,\label {lab1}\eea
where
$c_{m,2}=(\Gam (1/2 -m) (2m-1)!)^{-1}$ and the integral on the right-hand side of (\ref {lab1z}) is understood in the principal value sense.
\end{lemma}
\begin{proof} Both statements summarize known facts from  \cite [Chapter 1, Sec. 3] {GSh1}. For instance, {\rm (ii)} can be proved as follows. Using the equality
$$
|t|^{\a -1}=\frac{\Gam (\a)}{\Gam (\a+2m-1)}\, (|t|^{\a +2m-2}\, \sgn t)^{(2m-1)},
$$
we write the left-hand side of (\ref{lab1z}) in the form
$$
-\underset
{\a=1-2m}{a.c.}\,\frac{\Gam (\a)}{\Gam (\a+2m-1)\, \Gam ( \a/2)}\,  (|t|^{\a +2m-2}\, \sgn t, \vp^{(2m-1)}(t)).
$$
The latter yields the principal value integral
$$
\frac{1}{\Gam (1/2 -m) (2m-1)!} \, \intl_\bbr \frac{\vp^{(2m-1)}(t)}{t}\, dt,$$
which coincides with  (\ref{lab1}).
\end{proof}

\begin{lemma} \label {lem3}  Let $n>2, \; |h|< 1$. 

\noindent {\rm (i)} The integral
\be\label {den}
g_\a (h)=\frac{1}{\Gam (\a/2)} \intl_{-1}^1 |t \!-\!h|^{\a -1}\,(1-t^2)^{(n-3)/2}\, dt , \qquad  Re\, \a >0,
\ee
extends as an entire function of $\a$ and this extension represents a $C^\infty$ function of $h$ uniformly in $\a\in K$ for any compact subset $K$ of the complex plane.

\noindent {\rm (ii)}  Moreover,
\be\label {lab2} 
\underset
{\a=3-n}{a.c.}\, g_\a (h)= \Gam ((n-1)/2).\ee
 \end{lemma}

 The proof of this lemma is given in Appendix.

\section{The Euclidean case. Derivation of the Inversion Formula}\label {sec3}

We recall that our aim  is to reconstruct a $C^\infty$ function $f$ supported in
the  ball $B=\{x \in \rn: |x|<R\}$ provided that 
the spherical means
\[
(Mf)(\xi, t)=\frac{1}{\sig_{n-1}}\intl_{S^{n-1}} f (\xi
-t\sig)\, d\sig,\qquad (\xi, t) \in \partial B \times \bbr_+,\]
are known for all spheres centered on the boundary $\partial B $ of $B$ (Fig. 1).

\begin{figure}
\centerline{
\scalebox{0.6}{\includegraphics{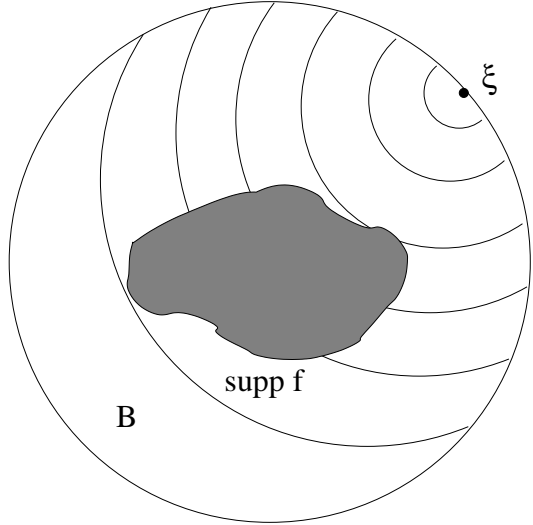}}}
\caption{The Euclidean case.}
\label{fig1}
\end{figure} 

We introduce the ``back-projection" operator $P$ that sends a function $F(\xi, t)$  on $\partial B \times \bbr_+$ to a function $(PF)(x)$ on $B$ by the formula
 \be\label{pfn}
(PF)(x)= \frac{1}{|\partial B|}\intl_{\partial B } F(\xi, |x-\xi|)\,
d\xi, \qquad x \in B,\ee
where $d\xi$ stands for the surface element of $\partial B$ and  $|\partial B|$ denotes the area of $\partial B$.

\subsection{The case $n>2$} Consider the following analytic family of operators
\be
(N^\a f)(\xi, t)=\intl_B \frac{|t^2 -|y-\xi|^2|^{\a -1}}{\Gam (\a/2)}\, f(y)\, dy, \ee
$$ 
(\xi, t) \in \partial B \times \bbr_+, \qquad Re \, \a >0.
$$

\begin{lemma}\label{aau} Let $f$ be an infinitely differentiable function
supported in $B=\{x \in \rn: |x|<R\}$. Then
\be\label {anal} \underset
{\a=3-n}{a.c.} (PN^\a
f)(x)=\lam_n\, \intl_B \frac {f(y)}{|x-y|^{n-2}}\,dy,\ee
\be\label {822}
\lam_n=(2R)^{2-n}\pi^{-1/2}\, \Gam (n/2).\ee
\end{lemma}

\begin{proof} For $Re \, \a >0$, changing the order of
integration, we obtain
$$ (PN^\a f)(x)=\intl_B f(y)\,k_\a (x,y)\, dy,
$$ where
 \bea k_\a (x,y)&=&\frac{1}{|\partial B|\, \Gam
 (\a/2)}\intl_{\partial B } 
 ||x\!-\!\xi |^2\!-\!|y\!-\!\xi |^2|^{\a -1}\, d\xi\nonumber\\
&=&
 \frac{1}{\sig_{n-1}\, \Gam
 (\a/2)}\intl_{S^{n-1}} 
||x|^2-|y|^2-2R\theta \cdot
(x-y)|^{\a -1}\, d\theta\nonumber\\
&=& \frac{(2R\,|x-y|)^{\a-1}}{\sig_{n-1}}\, \intl_{ S^{n-1}}\frac{|\theta \cdot \sig  -h|^{\a -1}}{\Gam (\a/2)}\, d\theta,
\label{pok}\eea
$$\sig=\frac{x-y}{|x-y|}, \qquad h=\frac{|x|^2-|y|^2}{2R\, |x-y|}.$$
By the rotation invariance of the inner product, the integral  in (\ref{pok}) is independent of $\sig$ and can be written as
$$
\frac{\sig_{n-2}}{\Gam (\a/2)} \intl_{-1}^1  |t \!-\!h|^{\a -1}\,(1-t^2)^{(n-3)/2}\, dt=\sig_{n-2}\,g_\a (h);$$
cf. (\ref{den}).  Note that $|h|< 1-\del$ for some $\del>0$ because $x$ and $y$ belong to the support of $f$ and the latter is separated from the boundary of $B$.  Hence,   Lemma \ref{lem3}   yields
\bea\label {kljm}
 &&\underset
{\a=3-n}{a.c.} \,(PN^\a
f)(x)=\frac{(2R)^{2-n}\, \sig_{n-2}}{\sig_{n-1}}\intl_B \frac {f(y)}{|x-y|^{n-2}}\, \underset
{\a=3-n}{a.c.} \, g_\a (h)\, dy\nonumber\\
&&=\lam_n\, \intl_B \frac {f(y)}{|x-y|^{n-2}}\,dy,\qquad \lam_n=(2R)^{2-n}\pi^{-1/2}\, \Gam (n/2)\nonumber\eea
(to justify interchange of integration and analytic continuation, the reader may consult, e.g., \cite[Lemma 1.17 ]{Ru96}).
\end{proof}

Let us obtain another representation of the analytic continuation  of $PN^\a f$, now, in terms of the spherical means of $f$.

\begin{lemma}\label{aau2} Let  $f$ be an infinitely differentiable function
supported in $B=\{x \in \rn: |x|<R\}$,
\be\label{amyw}
D=\frac {1}{2t}\frac{d}{dt}, \qquad \del_{n}=
\frac{(-1)^{[n/2-1]}\, \Gam((n-1)/2)}{(n\!-\!3)!}.\ee

\noindent {\rm (i)} If $n=3,5, \ldots\, $, then

\be\label {anx} \underset
{\a=3-n}{a.c.} (PN^\a
f)(x)=\frac{\del_{n}}{2R^{n-1}} \,  \intl_{\partial B} D^{n-3} [t^{n-2} (Mf)(\xi, t)]\Big |_{t=|x-\xi|}\, d\xi.
\ee

\noindent {\rm (ii)} If $n=4,6,\ldots\,$, then

\be\label {anx2} \underset
{\a=3-n}{a.c.} (PN^\a
f)(x)=-\frac{\del_{n}}{\pi\,R^{n-1}}  \intl_{\partial B}  d\xi
\intl_0^{2R} t\,D^{n-2} [t^{n-2} (Mf)(\xi, t)]\,\log |t^2-|x-\xi|^2|\, dt.\ee
\end{lemma}
\begin{proof} Passing to polar coordinates, we have
\bea
&&(N^\a f)(\xi, t)=\sig_{n-1}\intl_0^{2R}\frac{|t^2 -r^2|^{\a -1}}{\Gam (\a/2)}\, (Mf)(\xi, r)\, r^{n-1}\, dr\nonumber\\
&&=\intl_0^{4R^2}\frac{|t^2 -\t|^{\a -1}}{\Gam (\a/2)}\,\vp_\xi (\t)\, d\t,\quad \vp_\xi (\t)=
\frac{\sig_{n-1}}{2}\, \t^{n/2 -1} (Mf)(\xi, \t^{1/2}).\nonumber\eea
Since the support of $f$ is separated from the boundary of $B$,  there is an $\e>0$ such that $\vp_\xi (\t) \equiv 0$ when
$\t \notin (\e, 4R^2-\e)$. Hence, $\vp_\xi (\t)$ can be regarded as a function in $C_c^\infty (\bbr)$ and we can write
$$
(N^\a f)(\xi, t)= \intl_\bbr \frac{|\t|^{\a -1}}{\Gam ( \a/2)} \,  \vp_\xi(\t+t^2)\, d\t.$$
Now, Lemma \ref{lem1} yields
 the following equalities.

\noindent For $n=3,5, \ldots\, $:
$$
 \underset
{\a=3-n} {a.c.}\, (N^\a f)(\xi, t)\!=\!\del_{n,1} \,  \vp_\xi^{(n-3)} (t^2),\quad \del_{n,1}\! =\!\frac{(-1)^{(n-3)/2}\, ((n\!-\!3)/2)!}{(n\!-\!3)!}.
$$
\noindent For $n=4,6,\ldots\,$:
$$
 \underset
{\a=3-n}{a.c.}\,(N^\a f)(\xi, t)\!=\!\del_{n,2} \,\intl_\bbr \!\!\vp_\xi^{(n-2)} (\t)\,\log |\t\!-\!t^2|\, d\t,$$
$$ \del_{n,2}=-\frac{1}{\Gam ((3\!-\!n)/2)\,(n\!-\!3)!}.
$$
Combining these formulas with the backprojection $P$ and noting that operations
$ a.c.$ and $P$ commute,  we obtain

\vskip 0.3 truecm

\noindent For $n=3,5, \ldots\, $:
$$
  \underset
{\a=3-n}{a.c.} (PN^\a
f)(x)=\frac{\del_{n,1}}{|\partial B|} \, \intl_{\partial B } \vp_\xi^{(n-3)} (|x-\xi|^2)\, d\xi.
$$
\noindent For $n=4,6,\ldots\,$:
$$
 \underset
{\a=3-n}{a.c.}\,(PN^\a
f)(x)=\frac{\del_{n,2}}{|\partial B|}   \intl_{\partial B }  d\xi
\intl_0^{4R^2} \vp_\xi^{(n-2)} (\t)\,\log |\t\!-\!|x-\xi|^2|\, d\t.$$
These formulas give the desired result.
\end{proof}

Comparing different forms of  the analytic continuation in Lemmas \ref{aau} and \ref{aau2}, we obtain the following  statement.

\begin{lemma}\label{nxz} Let  $f$ be an infinitely differentiable function
supported in the unit ball $B=\{x \in \rn: |x|<R\}$, $D=\frac {1}{2t}\frac{d}{dt}\,$. Then
\be\label {basic}
\lam_n\, \intl_B \frac {f(y)}{|x-y|^{n-2}}\,dy\qquad   \qquad \qquad \qquad \qquad \qquad \qquad \qquad \qquad \ee
\[
=\left\{ \!
 \begin{array} {l} \displaystyle{\frac{\del_{n}}{2R^{n-1}} \,  \intl_{\partial B} D^{n-3} [t^{n-2} (Mf)(\xi, t)]\Big |_{t=|x-\xi|}\, d\xi}, \\ \mbox{if $\;n\!=\!3,5,\ldots \,$,}\\
 {}\\
 \displaystyle{-\frac{\del_{n}}{\pi\, R^{n-1}}  \intl_{\partial B}  d\xi
\intl_0^{2R} t\,D^{n-2} [t^{n-2} (Mf)(\xi, t)]\,\log |t^2-|x-\xi|^2|\, dt},\\ \mbox{if $\;n\!=\!4,6,\ldots \,$,}\\
 \end{array}
\right.\]
where $\lam_n$ and $\del_{n}$
 are defined by (\ref{822}) and (\ref{amyw}), respectively.
\end{lemma}

The left-hand side of (\ref{basic}) is a constant multiple of  the Riesz potential of order $2$ defined by 
\be\label{rpo}(I^2f)(x)=\frac{\Gam (n/2 -1)}{4\pi^{n/2}}\int_B
\frac{f(y)\, dy}{|x-y|^{n-2}}\ee and satisfying
\be\label{rpoi}
-\Del I^2f =f, \qquad \Del
=\sum\limits_{k=1}^n\frac{\partial^2}{\partial x_k^2}.
\ee
Thus, we arrive at the following result.

\begin{theorem}  \label{mku} Let  $f$ be an infinitely differentiable function
supported in  the unit ball $B=\{x \in \rn: |x|<R\}$, $D=\frac {1}{2t}\frac{d}{dt}\,$.

\noindent {\rm (i)} If $n=3,5, \ldots\, $, then

\be\label {inv1}
f(x)=d_{n,1} \, \Del \intl_{\partial B} D^{n-3} [t^{n-2} (Mf)(\xi, t)]\Big |_{t=|x-\xi|}\, d\xi,
\ee
$$
d_{n,1}=\frac{(-1)^{(n-1)/2}\, \pi^{1-n/2}}{4R\,\Gam(n/2)}.$$

\noindent {\rm (ii)} If $n=4,6,\ldots\,$, then

\be\label {inv2} f(x)\!=\!d_{n,2}\, \Del \intl_{\partial B}  d\xi
\intl_0^{2R} t\,D^{n-2} [t^{n-2} (Mf)(\xi, t)]\,\log |t^2-|x-\xi|^2|\, dt,\ee
$$
d_{n,2}=\frac{(-1)^{n/2-1}\, \pi^{-n/2}}{2R\, (n/2 -1)!}.$$
\end{theorem}

\subsection{The case $n=2$}  Let $\D$ be the open disk in $\bbr^2$ of radius $R$ centered at the origin. 
In this section, for the sake of completeness, we reproduce  (with minor changes) the argument from \cite{FHR}, keeping in mind that  the Riesz potential of order $2$ in the previous section is substituted by the logarithmic potential
\begin{equation}
(I_{\ast}f)(x)=\frac{1}{2\pi}\int\limits_{\D}f(y)\log|x-y|\,dy\,,\label{n2.2}%
\end{equation}
satisfying $\Delta I_{\ast}f=f$. 

The following statement is a substitute for 
 Lemma \ref{lem3}.\smallskip

\begin{lemma}\label {ooim}
\label{Lemma n.1}Let $-1<h<1,\;\sigma\in S^{1}$. Then 
\begin{equation}
g_{\ast}\equiv\int\limits_{S^{1}}\log|\theta\cdot\sigma-h|\,d\theta=-2\pi\log2.
\label{n2.3}%
\end{equation}
\end{lemma}

\begin{proof}
Owing to  rotational invariance,  we can write%
\begin{equation}
g_{\ast}=2\int_{-1}^{1}\frac{\log\left\vert t-h\right\vert }%
{\sqrt{1-t^{2}}}\,dt\,.\label{n2.5}%
\end{equation}
This integral is known; see, e.g.,    \cite[p. 296]{EK}, \cite{FHR}.\footnote{A more general integral was evaluated in \cite [Lemma 6.1]{AR}.} \end{proof}

\begin{lemma}\label{nn7} Let $f$ be a $C^\infty$  function supported
in $\D$. Then
\be\label{kuku}
(I_{\ast}f)(x)\!=\!\frac{1}{2\pi R}\intl_{\partial\D
}\int\limits_{0}^{2R}\!\left(Mf\right)  \left(\xi,t\right)  \log\left\vert
t^{2}\!-\!\left\vert x\!-\!\xi\right\vert ^{2}\right\vert \,t\,dt\,d\xi\!+\!c_f,
\ee
$$
c_f=-\frac{\log R}{2\pi }\intl_{\D} f(y)\, dy.$$
\end{lemma}
\begin{proof}  Let
$$
\left( N_{\ast}f\right)  \left(\xi,t\right) \! =\! \int\limits_{\D}\! f\left(
y\right)  \log\left\vert t^{2}\! -\! \left\vert y\! -\! \xi\right\vert ^{2}\right\vert
\,dy.$$
Changing the order of
integration and making use of (\ref{n2.3})   with 
$$\sig=\frac{x-y}{|x-y|}, \qquad h=\frac{|x|^2-|y|^2}{2R\, |x-y|},$$
 we obtain
$$ (PN_{\ast} f)(x)=\intl_\D f(y)\,k_{\ast} (x,y)\, dy,
$$ where
 \bea k_{\ast} (x,y)&=&\frac{1}{2\pi R}\intl_{\partial \D } 
 \log\left\vert \left\vert
x-\xi\right\vert ^{2}-\left\vert y-\xi\right\vert ^{2}\right\vert
\,d\xi\nonumber\\
&=&\frac{1}{2\pi R}\intl_{\partial \D } 
\log\left\vert \left\vert
x\right\vert ^{2}-\left\vert y\right\vert ^{2}-2\xi \cdot\left(  x-y\right)
\right\vert \,d\xi\nonumber\\
&=&\frac{1}{2\pi}\int\limits_{S^{1}}\left( \log (2R\left\vert x-y\right\vert) +\log\left\vert
h-\theta\cdot\sig\right\vert \right)  \,d\theta\nonumber\\
&=&\log (2R\left\vert x-y\right\vert) +\frac{g_{\ast}}{2\pi}=\log R+\log |x-y|.
\nonumber\eea
This gives
\be\label{90k}
(PN_{\ast} f)(x)=\intl_\D f(y)\,\log |x-y|\, dy+\log R \,\intl_\D f(y)\, dy.\ee
On the other hand, $ (PN_{\ast} f)(x)$ can be expressed in terms of the spherical means. Indeed, passing to polar coordinates, we have
$$(N_{\ast} f) \left(\xi,t\right)=2\pi \int\limits_{0}^{2R}\left(Mf\right) 
 \left(\xi,r\right)  
\log | r^{2}- t^{2}|\, rdr $$
and therefore,
\be\label{se4}
(PN_{\ast} f)(x)=\frac{1}{R}\intl_{\partial\D
}\int\limits_{0}^{2R}\!\left(Mf\right)  \left(\xi,r\right)  \log\left\vert
r^{2}\!-\!\left\vert x\!-\!\xi\right\vert ^{2}\right\vert \,r\,dr\,d\xi.\ee
Comparing (\ref{90k}) and (\ref{se4}), we arrive at (\ref{kuku}).
\end{proof}

Lemma \ref{nn7}  allows us complete  Theorem \ref{mku} in the following
way.$\smallskip$

\begin{theorem}
\label{Theorem n2.1} Let $f$ be an infinitely differentiable function supported
in the disk $\D=\{x\in\mathbb{R}^{2}:|x|<R\}$. Then%
\begin{equation}\label {fin}
f\left(x\right)  =\Delta \left(  \frac{1}{2\pi R} \intl_{\partial \D
}\int\limits_{0}^{2R}\!\left(Mf\right)  \left(\xi,t\right)  \log\left\vert
t^{2}\!-\!\left\vert x\!-\!\xi\right\vert ^{2}\right\vert \,t\,dt\,d\xi\right).
\end{equation}
\end{theorem}

Formula  (\ref{fin})  can be formally obtained from  (\ref {inv2})  by setting $n=2$. It 
coincides with formula (1.4)   in \cite {FHR}.

\subsection {Modified inversion formulas}  We can replace $f$ by $\Del f$ in   (\ref{basic})
 Since $f$ is smooth and $\supp f$ is separated from the
boundary of $B$, then $I^2 \Del f=-f$. Furthermore, since
$u(x,t)\equiv (M f)(x, t)$ satisfies the  Darboux equation 
$$
\square u \equiv \Del u - u_{tt} -\frac{n-1}{t}\,u_{t}=0
$$
 and $\Del$ commutes with rotations and
translations, then 
$$
(M \Del f)(x,t)= (\Del M f)(x,t)=L [(M f)(x,\cdot)](t), \qquad \forall x \in \rn, \quad t>0,$$
where
$$
 L=\frac{d^2}{dt^2}+\frac{n-1}{t}\,\frac{d}{dt};
 $$
  see, e.g., \cite[p. 17]{He}. This reasoning and its obvious analogue for $n=2$ give the following modifications of  inversion formulas  (\ref {inv1}),  (\ref {inv2}),  and   (\ref{fin})  with the same constant factors:

\noindent {\rm (i)} If $n=3,5, \ldots\, $, then

\be\label {inv1m}
f(x)=d_{n,1} \,  \intl_{\partial B} D^{n-3} [t^{n-2} (LMf)(\xi, t)]\Big |_{t=|x-\xi|}\, d\xi.
\ee

\noindent {\rm (ii)} If $n=2,4,6,\ldots\,$, then

\be\label {inv2m} f(x)\!=\!d_{n,2}\,  \intl_{\partial B}  d\xi
\intl_0^{2R} t\,D^{n-2} [t^{n-2} (LMf)(\xi, t)]\,\log |t^2-|x-\xi|^2|\, dt.\ee

Formula (\ref{inv1m}) agrees  with  \cite[formula (3.8)] {Ru00a}.

\section{Spherical means and EPD equations}

Consider the Cauchy problem for the 
Euler-Poisson-Darboux equation:
\begin{equation}
\square_{\a}u\equiv\Delta u-u_{tt}-\frac{n+2\alpha-1}{t}\,u_{t}%
=0\,,\label{EPD.1}%
\end{equation}%
\begin{equation}
u(x,0)=f(x),\qquad u_{t}(x,0)=0.\label{EPD.2}%
\end{equation}
As in the previous section, we assume  that  $f$ is a smooth function with compact support in the ball $B=\{x\in \bbr^n:\, |x|<R\}$.  If  $\alpha\geq(1-n)/2$,  then (\ref {EPD.1})-(\ref {EPD.2})
has a unique
solution \be\label{ijt} u(x,t)=(M^{\a}f)(x,t) , \qquad x\in\bbr^n, \quad t\in \bbr_+,\ee where  $M^{\a}f$ is defined
as  analytic continuation of the integral  \[
(M^{\a}f)(x,t)  \!=\!\frac{\Gamma\left(  \alpha\!+\!n/2\right)  }{\pi^{n/2}\Gamma\left(
\alpha\right)  }\intl_{|y|<1}\!\!(1\!-\!|y|^{2})^{\alpha-1}%
f(x\!-\!ty)\,dy,  \qquad Re
\,\alpha\!>\!0;\]
see \cite{B2} for details. 
If  $\alpha=0$, then   $M^{0}f\equiv \underset
{\a=0}{a.c.} M^{\a}f$  represents  the spherical mean (\ref{mf}).

\smallskip

Consider the following  problem:

\textit{ Given   the trace $u\left( \xi,t\right)  $
 of  the solution  of  (\ref{EPD.1}) - (\ref{EPD.2})  for all  $\left(\xi,t\right)  \in\partial B\times\mathbb{R}_{+}$,  reconstruct $f(x)$. }

\smallskip

To solve this problem we need some  facts from fractional calculus; see,
e.g., \cite[Sec. 18.1]{SKM} or \cite[Sec. 9.6]{EK}.
 For $Re \,\alpha>0$ and
$\eta\geq-1/2$, the Erd\'{e}lyi-Kober fractional integral of a function
$\varphi$ on $\mathbb{R}_{+}$ is defined by
\begin{equation}
(I_{\eta}^{\alpha}\varphi)(t)=\frac{2t^{-2(\alpha+\eta)}}{\Gamma(\alpha)}%
\int_{0}^{t}(t^{2}-r^{2})^{\alpha-1}r^{2\eta+1}\varphi(r)\,dr,\qquad
t>0.\label{eci}%
\end{equation}
In our case it suffices to assume that $\varphi$ is infinitely smooth and
supported away from the origin. Then $I_{\eta}^{\alpha}\varphi$ extends as an
entire function of $\alpha$ and $\eta$, so that  $I_{\eta}^{0}\varphi=\varphi$, 
$\; \left( I_{\eta}^{\alpha}\right)^{-1}\!\varphi=I_{\eta+\alpha
}^{-\alpha}\varphi$, 
\[ (I_{\eta}^{-m}\varphi)(t)=t^{-2(\eta-m)}D^{m}\,t^{2\eta}\varphi(t),\qquad 
D=\frac{1}{2t}\,\frac{d}%
{dt}.\]

Assuming $x=\xi\in \partial B$ and passing to polar coordinates,  we obtain 
\begin{equation}\label{EDP.5}
u_\xi(t)=(M^{\a}f)(\xi,t)  =\frac{\Gamma\left(
\alpha+n/2\right)}{\Gamma\left(n/2\right)  }\left(  I_{\eta}^{\alpha}\varphi_\xi\right)  \left(  t\right),
\end{equation}
where  $u_{\xi}\left( t\right)  =u(\xi,t) $,  $\varphi_{\xi}\left(  t\right)  =(Mf)(\xi,t) $, $\eta=n/2-1$.
This gives
\begin{equation}\label{EDP.6}
\varphi_{\xi}=\frac{\Gamma(n/2)}{\Gamma(\alpha+n/2)}\left(  I_{\eta}^{\alpha}\right) ^{-1}u_{\xi}=\frac{\Gamma(n/2)}{\Gamma(\alpha+n/2)}I_{\eta+\alpha
}^{-\alpha}u_{\xi}.
\end{equation}
Now, since $\varphi_{\xi}(t)=(Mf)(\xi,t) $ is known, we can use
 Theorem \ref{mku}  to reconstruct  $f$ by the following formulas.

\begin{theorem}\label{Th4.1} 
 Let  $f$ be an infinitely differentiable function
supported in  the unit ball $B=\{x \in \rn: |x|<R\}$, $D=\frac {1}{2t}\frac{d}{dt}\,$.

\noindent\textrm {\rm (i)}  If $n=3,5,\ldots\,$, then%
\[
f(x)=\widetilde{d}_{n,1}\,\Delta\int\limits_{\partial B}D^{n-3}[t^{n-2}%
(I_{\eta+\alpha}^{-\alpha}u_{\xi})\left( t\right)  ]\Big |_{t=|x-\xi|}%
\,d\xi,
\]
\[
\widetilde{d}_{n,1}=\frac{(-1)^{(n-1)/2}\,\pi^{1-n/2}}{4R\,\Gamma(\alpha
+n/2)}.\]

\noindent\textrm {\rm (ii)}  If $n=2,4,6,\ldots\,$, then%
\[f(x)=\widetilde{d}_{n,2}\,\Delta\int\limits_{\partial B}d\xi\int
\limits_{0}^{2R}t\,D^{n-2}[t^{n-2}(I_{\eta+\alpha}^{-\alpha}u_{\xi})\left(
t\right)  ]\,\log|t^{2}-|x-\xi|^{2}|\,dt,\]
\[
\widetilde{d}_{n,2}=\frac{(-1)^{n/2-1}\,\pi^{-n/2}}{2R\,\Gamma(\alpha
+n/2)}.\]
\end{theorem}

The case $\a=(1-n)/2$ in this theorem gives explicit solution to the TAT problem (see Introduction) 
 with constant speed $c(x)\equiv 1$. Moreover,  after  $f$ has been  found, we can reconstruct $u(x,t)$ in the whole space by setting $u(x,t)=(M^{\a}f)(x,t)$. The latter gives an explicit solution to  the Cauchy problem for the generalized 
EPD equation (\ref{EPD.1}) with initial data on the cylinder $\partial B\times\mathbb{R}_{+}$.

\section{Spherical Means on $S^n$}

  The suggested method of analytic continuation enables us to study the spherical mean Radon transform $Mf$ on  arbitrary constant curvature  space $X$. In this setting,  $\supp f \subset B$,   
  where $B$ is a geodesic ball  centered at the origin, and the spherical means of $f$ are evaluated over  geodesic spheres, the centers of which are located on the boundary of $B$.  In this section we consider the case, when $X=S^n$ is the $n$-dimensional sphere  in $\bbr^{n+1}$. 
  
  Given  $x\in S^n$ and $t \in (-1,1)$, let 
\be \label {77b}
(Mf)(x, t)=\frac{(1-t^2)^{(1-n)/2}}{\sig_{n-1}}\intl_{x\cdot y=t} f (y)\, d\sig (y) \ee
be the mean value of a function $f\in C^\infty (S^n)$  over the planar section  $\{y\in S^n: x\cdot y=t\}$  (Fig.2).

Our aim is to reconstruct $f$ under the following assumptions:

\noindent (a) The support of $f$ lies on the spherical cap
\be B_\theta=\{x\in S^n:  x\cdot e_{n+1} >\cos \,\theta\},\ee
(the geodesic ball of
 radius $\theta$), where $e_{n+1}=(0, \ldots, 0,1)$ is the north pole of $S^n$ and $\theta \in (0,\pi/2]$ is fixed.

\noindent (b) The mean values (\ref{77b}) are known for all  $x=\xi\in  \partial B_\theta $ and all $t\in  (-1,1)$, where
$\partial B_\theta$ is the boundary of  $B_\theta$.

\begin{figure}
\centerline{
\scalebox{0.6}{\includegraphics{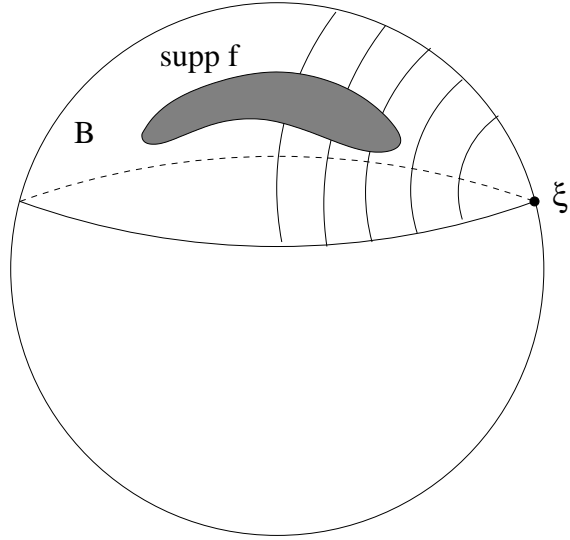}}}
\caption{The spherical case.}
\label{fig2}
\end{figure}

This problem can be solved using the method of the previous section.  Let us fix our notation. In the following $S^n_+=\{x
  \in S^n  :  x_{n+1}\ge 0\}$ is the upper hemisphere,  $B_0$ denotes the unit ball in the  hyperplane $ x_{n+1}=0$;  $S^{n-1}$ stands for the boundary of $B_0$, which is also the boundary of  $S^n_+$.  For $x\in S^n_+$ we write
$$
x=(x',  \sqrt{1-|x'|^2}),  \qquad x'=(x_1, \ldots, x_n,0) \in B_0,$$
so that
\bea 
\intl_{S^n_+} f(x) \, dx&=&\intl_{B_0} f(x',  \sqrt{1-|x'|^2}) \sqrt{1+
\left (\frac{\partial x_{n+1}}{\partial x_1}\right )^2+\ldots +\left (\frac{\partial x_{n+1}}{\partial x_n}\right )^2}\, dx'\nonumber\\
&=&\intl_{B_0} f(x',  \sqrt{1-|x'|^2}) \,\frac{dx'}{ \sqrt{1-|x'|^2}}. \nonumber\eea

We  introduce the  backprojection operator $P$,  that sends  functions on $\partial B_\theta \times  (-1,1)$ to  functions on $B_\theta$ by the formula
 \be\label{pfns}
(PF)(x)= \frac{1}{|\partial  B_\theta|}\,\intl_{\partial B_\theta} F(\xi, \xi\cdot x)\,
d\xi, \qquad x \in B_\theta.\ee
Here   $d\xi$  and  $|\partial  B_\theta|$ denote the surface element and the area of  $\partial B_\theta $, respectively.   We denote by $\tilde B_\theta$ the orthogonal projection of $B_\theta$ onto 
the  hyperplane $ x_{n+1}=0$.

\subsection{The case $n>2$} \label {98m} Assuming $(\xi, t)\in\partial B_\theta \times  (-1,1)$ and $Re \, \a >0$, consider the following analytic family of operators
\be
(N^\a f)(\xi, t)=\intl_{B_\theta} \frac{|\xi\cdot y -t|^{\a -1}}{\Gam (\a/2)}\, f(y)\, dy.\ee
\begin{lemma}\label{aaus}  Let $f\in C^\infty (S^n)$, $\supp f \subset B_\theta$. Then
\be\label {anals} \underset
{\a=3-n}{a.c.} (PN^\a
f)(x)=\frac{\Gam (n/2)\, (\sin \theta)^{2-n}}{\pi^{1/2}}\intl_{\tilde B_\theta}  \frac {\tilde f(y')\, dy'}{|x'-y'|^{n-2}},\ee
\be
\tilde f(y')=(1-|y'|^2)^{-1/2}\,f(y', (1-|y'|^2)^{1/2}).
\ee
\end{lemma}

\begin{proof} For $Re \, \a >0$, changing the order of
integration, we obtain\footnote{For the sake of convenience, we use some notations which mimic  analogous expressions in the Euclidean case. }
$$
(PN^\a f)(x)=\intl_{B_\theta} f(y)\, k_\a (x,y), \, dy, \quad k_\a (x,y)=\frac{1}{|\partial  B_\theta|}\,\intl_{\partial B_\theta} 
\frac{|\xi\cdot (x- y)|^{\a -1}}{\Gam (\a/2)}
d\xi.$$
Since $\xi$ has the form $\xi=e_{n+1}\, \cos\, \theta +\om \sin\, \theta, \; \om \in S^{n-1}$, then
\bea
|\xi\cdot (x- y)|&=&|(x_{n+1}-y_{n+1})\, \cos\,\theta+(x'- y')\cdot \om\, \, \sin\theta |\nonumber\\
&=& \left | h-\om \cdot \sig \right |\, |x'- y'|\, \sin\theta, \label{laac}\eea
\be h=\frac{x_{n+1}-y_{n+1}}{|x'- y'|}\, \cot \theta, \qquad \sig=\frac{x'- y'}{|x'- y'|}.\ee
Hence,
\be\label{okm}
k_\a (x,y)=\frac {(|x'- y'|\, \sin\theta)^{\a-1}}{\sig_{n-1}}\,\intl_{S^{n-1}} \frac{|h-\om \cdot \sig |^{\a -1}}{\Gam (\a/2)}\, d\om.\ee
The integral  in (\ref{okm}) is independent of $\xi$ and can be written as
$$
\frac{\sig_{n-2}}{\Gam (\a/2)} \intl_{-1}^1  |t \!-\!h|^{\a -1}\,(1-t^2)^{(n-3)/2}\, dt=\sig_{n-2}\,g_\a (h);$$
cf. (\ref{den}).  This gives
\be\label{jui}
k_\a (x,y)=\frac {\sig_{n-2}\,(|x'- y'|\, \sin\theta)^{\a-1}}{\sig_{n-1}}\, g_\a (h) .\ee

Let us show that $|h|<1$. We write
$$
x=e_{n+1}\, \cos\, \gam +u \sin\, \gam, \qquad y=e_{n+1}\, \cos\, \del +v\sin\,\del, $$
$$
\gam,  \del \in (0,\theta); \qquad  u,v\in S^{n-1}; \qquad \tilde h=\frac{x_{n+1}-y_{n+1}}{|x'- y'|}.$$
Then
\bea
 |\tilde h|^2&=&\frac{(\cos\, \gam -\cos\, \del)^2}{|u \sin \gam - v\sin\del |^2}\nonumber\\
 &=&\frac{(\cos\, \gam -\cos\, \del)^2}{\sin^2 \gam - 
 2(u\cdot v)\sin \gam \,\sin\del+\sin^2 \del}\nonumber\\
 &\le&\frac{(\cos\, \gam -\cos\, \del)^2}{\sin^2 \gam - 
 2\sin \gam \,\sin\del+\sin^2 \del}=
 \frac{(\cos\, \gam -\cos\, \del)^2}{(\sin \gam - \sin \del )^2}. 
 \nonumber\eea
 Without loss of generality, suppose that $\gam\le \del$. Then
 $$
  |\tilde h|\le \frac{\cos\, \gam -\cos\, \del}{\sin \del -\sin \gam}=\tan  \frac{\gam +\del}{2}<\tan \theta,
$$
  and therefore, $|h|= |\tilde h| \,  \cot \theta<1$. 

Since $|h|<1$,  Lemma \ref{lem3}  yields
$$
 \underset
{\a=3-n}{a.c.} (PN^\a
f)(x)=c_n \intl_{B_\theta}\frac {f(y)\, dy}{|x'-y'|^{n-2}}=
c_n \intl_{\tilde B_\theta}  \frac {\tilde f(y')\, dy'}{|x'-y'|^{n-2}},$$
$$
\tilde f(y')\!=\!(1\!-\!|y'|^2)^{-1/2}\,f(y', (1\!-\!|y'|^2)^{1/2}), \quad c_n\!=\!
\frac{\Gam (n/2)\, (\sin \theta)^{2-n}}{\pi^{1/2}}.
$$
\end{proof}

As before, we need one more representation of  $\underset
{\a=3-n}{a.c.} (PN^\a
f)(x)$, now in terms of the spherical means $(Mf)(\xi, t)$.

\begin{lemma}\label{aauss}    Let $f\in C^\infty (S^n)$, $\supp f \subset B_\theta$.  Then
\be\label {anxs} \underset
{\a=3-n}{a.c.} (PN^\a
f)(x)=\frac{\del_{n}}{(\sin\theta)^{n-1}}\,  \intl_{\partial B_\theta} (d/dt)^{n-3} [(Mf)(\xi, t)\, (1-t^2)^{n/2 -1}]\Big |_{t=\xi\cdot x}\, d\xi,\nonumber\ee
if  $n=3,5, \ldots\, $, and 
\bea\label {anx2s} &&\underset
{\a=3-n}{a.c.} (PN^\a
f)(x)=-\frac{\del_{n}}{\pi\,(\sin\theta)^{n-1}} \,  \intl_{\partial B_\theta}  
d\xi\nonumber\\
&&\times\intl_{\cos\, 2\theta}^1  (d/dt)^{n-2} [(Mf)(\xi, t)\, (1-t^2)^{n/2 -1}]\,\log |t-\xi\cdot x|\,dt,\nonumber\eea
if $n=4,6,\ldots\,$, where $\del_{n}$ is defined by (\ref{amyw}).

\end{lemma}
\begin{proof} For $Re \, \a >0$, by making use of the formula
\be\label{8ghv}
\intl_{S^n} f(y)\, a (\xi\cdot y) \, dy=\sig_{n-1}\intl_{-1}^1 a(\t) (Mf)(\xi, \t)\,(1-\t^2)^{n/2 -1}\, d\t,\ee
we have
\bea
&&(N^\a f)(\xi, t)=\frac{\sig_{n-1}}{\Gam (\a/2)}\intl_{-1}^1  (Mf)(\xi, \t)\,|\t -t|^{\a -1} (1-\t^2)^{n/2 -1}\, d\t\nonumber\\
&&= \intl_\bbr \frac{|\t|^{\a -1}}{\Gam ( \a/2)} \,  \vp_\xi (\t+t)\, d\t, \quad  \vp_\xi (\t)=\sig_{n-1}\,(Mf)(\xi, \t) \,(1-\t^2)^{n/2 -1}_+.\nonumber\eea
Since $f$ is smooth and its support is separated from the boundary 
 $\partial B_\theta$,  then $ (Mf)(\xi, \t)$ is smooth in the $\t$-variable uniformly in $\xi$  and vanishes identically in the respective neighborhoods of $\t=\pm 1$. Thus, we can invoke 
 Lemma \ref{lem1}  which yields the following equalities.

\noindent For $n=3,5, \ldots\, $; $\cos\, 2\theta<t<1$:
$$
 \underset
{\a=3-n} {a.c.}\,(N^\a f)(\xi, t)\!=\!\del_{n} \,  \vp_\xi^{(n-3)} (t),$$

\noindent For $n=4,6,\ldots\,$:
$$
 \underset
{\a=3-n}{a.c.}\,(N^\a f)(\xi, t)=-\frac{\del_{n}} {\pi}\,\intl_{\cos\, 2\theta}^1\!\!\vp_\xi^{(n-2)} (\t)\,\log |\t\!-\!t|\, d\t,$$
$\del_n$ being defined by (\ref{amyw}). 
The above formulas mimic those in the proof of Lemma \ref{aau2}  and the result follows.
\end{proof}

 Lemmas   \ref{aaus}  and   \ref{aauss} imply the following inversion result for the spherical means on $S^n$. In the statement below, $\Del_{x'}=\partial_1^2 +\ldots +\partial_n^2$ is the usual Laplace operator in the $x'$-variable.

\begin{theorem} \label{774}  Let $f\in C^\infty (S^n)$, $\supp f \subset  B_\theta$.  Then
\be\label{8njv} f(x)\!=\!  \frac{d_n\, x_{n+1}}{\sin\theta}\, \Del_{x'} f_0 (x', \sqrt{1\!-\!|x'|^2}),\quad d_{n}\!=\!\frac{(-1)^{[n/2-1]}}{2^{n-1} \pi^{n/2 -1}\Gam(n/2)},\ee
where  $f_0 (x)\equiv f_0 (x', \sqrt{1-|x'|^2})$ has the following form:
$$
f_0 (x)=-\intl_{\partial B_\theta} (d/dt)^{n-3} [(Mf)(\xi, t)\, (1-t^2)^{n/2 -1}]\Big |_{t=\xi\cdot x}\, d\xi
$$
if $n=3,5, \ldots\, $, and
$$
f_0 (x)=\frac{1}{\pi} \intl_{\partial B_\theta} d\xi
\intl_{\cos\, 2\theta}^1  (d/dt)^{n-2} [(Mf)(\xi, t)\, (1-t^2)^{n/2 -1}]\,\log |t-\xi\cdot x|\,dt$$
if $n=4,6,\ldots\,$.
\end{theorem}

\subsection{The case $n=2$}\label{5477}
We keep the notation of  section  \ref{98m}. Let
\be\label {87f}
(I_{\ast}f)(x)\equiv (I_{\ast}f)(x',  \sqrt{1-|x'|^2})
=\frac{1}{2\pi}\intl_{B_\theta} f(y)\log|x'-y'|\,dy,\ee
so that 
\be\label{laz} \Delta_{x'} (I_{\ast}f)(x)=(1-|x'|^2)^{-1/2}\,f (x)=f (x)/x_3. \ee

\begin{lemma}\label{nn7s} Let $f$ be a $C^\infty$  function supported
in $B_\theta$. Then
\be\label{kuks}
(I_{\ast}f)(x)\!=\!\frac{1}{|\partial  B_\theta|}\intl_{\partial B_\theta
}\int\limits_{-1}^{1}\!\left(Mf\right)  \left(\xi,\t\right)  \log |\t-
\xi\cdot x |\,d\t\,d\xi\!+\!c_f,
\ee
$$
c_f=-\frac{1}{2\pi } \, \Big (\log \, \frac{\sin \theta}{2} \Big)\,\intl_{B_\theta} f(y)\, dy.
$$
\end{lemma}
\begin{proof}  Let
$$
\left( N_{\ast}f\right)(\xi, t) =\intl_{B_\theta} f(y)\log|\xi\cdot y-t|\,dy, \qquad (\xi, t)\in\partial B_\theta \times  (-1,1).
$$
Changing the order of
integration, owing to (\ref{laac}),  we obtain
$$ (PN_{\ast} f)(x)=\intl_{B_\theta} f(y)\,k_{\ast} (x,y)\, dy,
$$ where
 \bea k_{\ast} (x,y)&=& \frac{1}{|\partial  B_\theta|}\,\intl_{\partial B_\theta}  \log|\xi\cdot (x-y)|\,d\xi\nonumber\\
 &=& \frac{1}{2\pi}\,\intl_{S^1} [\log|x'-y'|+\log \, \sin \theta+\log  | h-\om \cdot \sig |]\, d\om\nonumber\\
 &=& \log|x'-y'|+\log \, \sin \theta +\frac{g_*}{2\pi},\nonumber\eea
  $$
g_* \equiv\int\limits_{S^{1}}\log  | h-\om \cdot \sig |\, d\om=2\int_{-1}^{1}\frac{\log\left\vert t-h\right\vert }
{\sqrt{1-t^{2}}}\,dt=-2\pi\log2;
$$
cf. Lemma \ref{ooim}.  This gives
\be\label{90ks}
(PN_{\ast} f)(x)=2\pi\,(I_{\ast}f)(x)+\Big (\log \, \frac{\sin \theta}{2} \Big)\,\intl_{B_\theta} f(y)\, dy.\ee
On the other hand, by (\ref{8ghv}), 
\be\label{m23}
(PN_{\ast} f)(x)=\frac{1}{ \sin \theta} \,\intl_{\partial B_\theta} \intl_{-1}^1 (Mf)(\xi, \t) \log |\t -\xi \cdot x|\, d\t d\xi.
\ee
Comparing (\ref{m23})  with (\ref{90ks}),  we obtain the result.
\end{proof}

Lemma \ref{nn7s}  allows us complete  Theorem \ref{774} in the following
way.$\smallskip$

\begin{theorem} 
\label{Theorem n2.1s2} Let $f$ be an infinitely differentiable function supported
in the spherical  cap $B_\theta 
=\{x\in S^2:  x\cdot e_{3} >\cos \,\theta\}$, $ \theta \in (0,\pi/2]$. Then
\begin{equation}\label {fins}
f\left(x\right)  =\frac{x_3}{2\pi\,  \sin \theta}\,\Delta_{x'}  \intl_{\partial B_\theta
}\int\limits_{-1}^{1}\!\left(Mf\right)  \left(\xi,\t\right)  \log |\t-
\xi\cdot x |\,d\t\,d\xi.
\end{equation}
\end{theorem}

Formula  (\ref{fins})  can be formally obtained from  (\ref {8njv})  by setting $n=2$. 

\section{The inverse problem for  the EPD
 equation on $S^n$}\label {45er}

 The Euler-Poisson-Darboux  equation on $S^n$ has the form
\be\label {papv}
\tilde\square_\a u\equiv \del_x u - u_{\om\om} - (n-1+2\a) \cot \om\, u_\om + \a (n-1+\a)u=0.
\ee
Here $x\in S^n$ is the space variable,  $\om\in (0,\pi)$ is the time variable,  $\del_x$ is the relevant Beltrami-Laplace operator. For the sake of simplicity, we restrict ourselves to the case  $Re \, \a>-n/2$. In this case the corresponding 
 Cauchy problem
 \be\label {papoz}  \tilde\square_\a u=0, \qquad u(x,0) = f(x), \quad  u_\om (x, 0) = 0,\ee
with   $f\in C^\infty  (S^n)$ has a solution  
$u(x,\om)=(M^\a f)(x, \cos\, \om)$, where $(M^\a f)(x, t) $ is defined as analytic continuation of the  integral
\be\label{nzb}
(M^\a f)(x,t) = \frac{c_{n, \a}}{ (1-t^2)^{\a-1+n/2}}\intl_{x\cdot y > t} (x\cdot
 y - t)^{\a -1} f(y) dy, 
\ee
$$
c_{n, \a} = 2^{\a-1} \pi^{-n/2} \Gamma (\a + n/2) /\Gamma(\a), \quad Re \,\a > 0, 
\quad   t \in (-1, 1); 
$$
see \cite {O2},  \cite[p. 179] {Ru00}, and references therein.

Let $ B_\theta=\{x\in S^n:  x\cdot e_{n+1} >\cos \,\theta\}$ be the spherical cap
 of a fixed
 radius  $\theta \in (0,\pi/2]$ and let $\partial B_\theta$ be the boundary of  $B_\theta$. Our aim is to solve the following 
 
 {\bf Inverse problem. } 
 {\it Suppose that the values  $g(\xi,\om)$  of the solution    of  
  (\ref{papoz})  are known for all $(\xi, \om)\in \partial B_\theta \times (0,\pi)$. Reconstruct the initial function $f\in C^\infty  (S^n)$, provided that the support of $f$ lies in  $B_\theta$.}

This problem can be solved using the results of the previous section. 
Assuming $Re \,\a > 0$, we pass to spherical polar coordinates and write (\ref{nzb}) as

 \[
 (M^\a f)(\xi, t) =  {c_{n, \a}\, \sig_{n-1}\over (1-t^2)^{\a-1+n/2}} \intl_{t}^1(\t - t)^{\a -1} (Mf)(\xi, \t)\,(1-\t^2)^{n/2 -1}\, d\t.\]
Then we set
$$
F_\xi (t)=(Mf)(\xi, t)\,(1-t^2)^{n/2 -1}, $$
$$
G_{\a,\xi} (t)=\frac{2^{1-\a} \pi^{n/2}}{ \Gamma (\a + n/2)\, \sig_{n-1}}\, (1-t^2)^{\a-1+n/2} g(\xi,\cos^{-1} t),$$
and  invoke
Riemann-Liouville fractional integrals \cite{SKM}
\be\label{rli} (I_-^\a u)(t)=\frac{1}{\Gam (\a)}\intl_{t}^1(\t - t)^{\a -1} u(\t)\, d\t, \qquad Re \,\a > 0.\ee
Thus, if $Re \,\a > 0$, then
 \be\label{opsa}
(I_-^\a F_\xi)(t)=G_{\a,\xi} (t).
\ee
Since $f$ is infinitely differentiable  and the support  of $f$  is separated from the boundary $\partial B_\theta$,  then $F_\xi$ is infinitely differentiable on $(-1,1)$ uniformly in $\xi$ and $\supp \,F_\xi$ does not meet the endpoints $\pm1$. It follows that 
(\ref{opsa}) extends by analyticity to all complex $\a$, and we have
$$
(Mf)(\xi, t)=(1-t^2)^{1-n/2} (I_-^{-\a} G_{\a,\xi})(t)
$$
where $I_-^{-\a}$  is understood in the sense of analytic continuation.
Now Theorem \ref{774}  yields the following  explicit solution of  our inverse problem. 

\begin{theorem}
Let  $ B_\theta$ be the spherical cap on $S^n$
 of  radius  $\theta \in (0,\pi/2]$,
$$
G_{\a,\xi}(t)\!=\!\frac{2^{1-\a} \pi^{n/2}}{ \Gamma (\a \!+\! n/2)\, \sig_{n-1}}\, (1-t^2)^{\a-1+n/2} g(\xi,\cos^{-1} t),\qquad Re \,\a\!>\!-n/2,$$
where  $g$ is a given function on $\partial B_\theta \times (0,\pi)$. If the initial function $f$ in the 
 Cauchy problem (\ref {papoz}) is  infinitely differentiable  and the  support of  $f$ lies in the interior of  $B_\theta$,  then
 $f$  can be reconstructed by the formula 
\be f(x)\!=\!  \frac{d_n\, x_{n+1}}{\sin\theta}\, \Del_{x'} f_0 (x', \sqrt{1\!-\!|x'|^2}),\quad d_{n}\!=\!\frac{(-1)^{[n/2-1]}}{2^{n-1} \pi^{n/2 -1}\Gam(n/2)},\ee
where  $f_0 (x)\equiv f_0 (x', \sqrt{1-|x'|^2})$ has the following form:
$$
f_0 (x)=-\intl_{\partial B_\theta} (d/dt)^{n-3} [(I_-^{-\a}  G_{\a,\xi})(t)]\Big |_{t=\xi\cdot x}\, d\xi
$$
if $n=3,5, \ldots\, $, and
$$
f_0 (x)=\frac{1}{\pi} \intl_{\partial B_\theta} d\xi
\intl_{\cos\, 2\theta}^1  (d/dt)^{n-2} [(I_-^{-\a}  G_{\a,\xi})(t)]\,\log |t-\xi\cdot x|\,dt$$
if $n=2,4,6,\ldots\,$.
\end{theorem}

We recall that the case $\a=(1-n)/2$ in this theorem gives a solution of the relevant inverse problem for the wave equation on $S^n$ in the framework of the formal spherical TAT model.

\section{Spherical Means in the Hyperbolic Space}
Most of the facts listed below can be found in \cite{VK2}. Let $\bbe^{n, 1}, \; n \ge 2$, be the real pseudo-Euclidean space
of points $x = (x_1, \dots, x_{n+1})$ with the inner product 
\be\label{201}
[x, y] = -x_1y_1-\dots - x_ny_n+x_{n+1} y_{n+1}.\ee The  hyperbolic
space $\bbh^n$ is interpreted as the ``upper" sheet of the two-sheeted
hyperboloid \be  \bbh^n = \{ x \in \bbe^{n, 1}: [x, x] = 1,
x_{n+1} > 0\}.\ee 
 The 
hyperbolic coordinates of a point 
$x = (x_1, \ldots, x_{n +1}) \in \bbh^n$ are defined by
\be\label {203}
 \left\{ \!
 \begin{array} {l} x_1 = \sinh r \sin \om_{n -1} \ldots 
\sin \om_2 \sin \om_1, \\
 x_2 = \sinh r \sin \om_{n -1} \ldots \sin \om_2 \cos \om_1, \\
 ...........................................................................................\\
x_n = \sinh r \,\cos \om_{n -1}, \\
x_{n +1} = \cosh  r ,\\
 \end{array}
\right.\ee
where 
$$0 \le \om_1 < 2 \pi; \quad 0 \le \om_j < \pi, \quad  1 < j \le n - 1; \quad  0 \le r < \infty.$$ By 
  (\ref{203}), each  $x \in \bbh^n$ can be represented as 
\be\label{oo90} x =  \om \,\sinh r  +  e_{n+1}\,\cosh r \ = (\om\,\sinh r, \cosh r)\ee where 
$\om$ is a point of the unit sphere $S^{n -1}$ in $\bbr^n$ with 
 Euler angles
 $\om_1, \ldots, \om_{n -1}$.  We regard $\bbr^n$ as the hyperplane $x_{n +1} = 0$ in $E^{n, 1}$. The 
 invariant  (with respect to hyperbolic motions) measure $dx $ in $ \bbh^n$ is given by 
$ d x = \sinh^{n -1} r \, d \om d r, $
 where $d \om$ is the surface element of $S^{n -1}$.  
 The geodesic distance between 
points $x$ and $ y$ in $\bbh^n$ is defined by $$
\dist (x,y)=\cosh^{-1} [x,y]\qquad \mbox{\rm (i.e.,  $\cosh\,\dist (x,y)=[x,y]$)}.$$ 
 Given  $x\in \bbh^n$ and $t>1$, let 
\be \label {77bh}
(Mf)(x, t)=\frac{(t^2-1)^{(1-n)/2}}{\sig_{n-1}}\intl_{[x, y] = t}  f (y)\, d\sig (y) \ee
be the mean value of  $f$  over the planar section  $\{y\in  \bbh^n: [x, y] = t\}$.

\begin{figure}
\centerline{
\scalebox{0.6}{\includegraphics{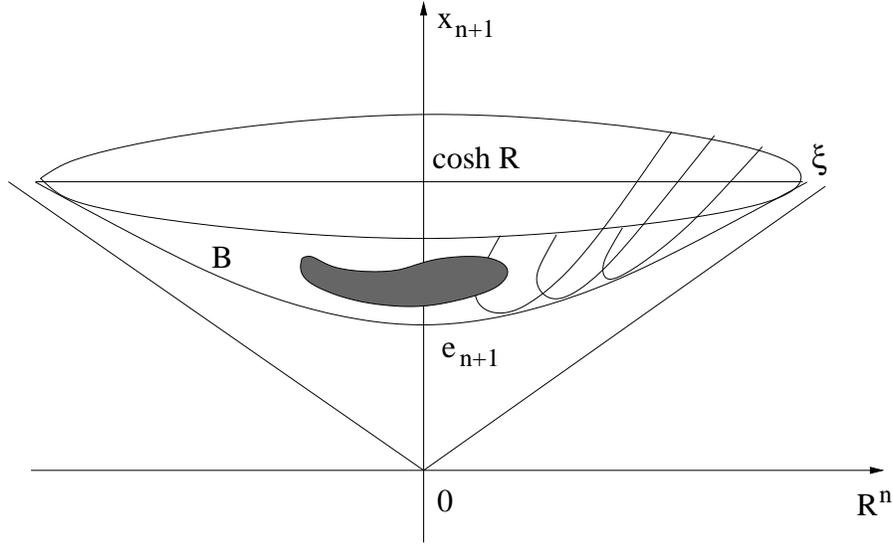}}}
\caption{The hyperbolic case.}
\label{fig3}
\end{figure}

 As before, our aim is to reconstruct a function $f\in C^\infty  (\bbh^n)$ under the following assumptions:

\noindent (a) The support of $f$ lies in the geodesic ball (Fig.3)
$$ B=\{x\in  \bbh^n: \dist (x,e_{n+1})<R\}=\{x\in  \bbh^n: x_{n+1} <\cosh R\},$$
where $e_{n+1}=(0, \ldots, 0,1)$ is the origin of $\bbh^n$ and $R>0$ is  fixed.

\noindent (b) The mean values (\ref{77bh}) are known for all   $x=\xi\in   \partial B $ and all $t>1$, where
$\partial B $ is the boundary of  $B$.

Owing to (\ref{oo90}), we can write $x\in \bbh^n$ as
$$
x=(x',  \sqrt{1+|x'|^2}),  \qquad x'=(x_1, \ldots, x_n,0) \in \bbr^n, $$
so that
\be\label {koj} 
\intl_{\bbh^n} f(x) \, dx\!=\!\intl_{\bbr^n} f(x',  \sqrt{1\!+\!|x'|^2}) \,\rho (x')\, dx', \quad  \rho (x')\!=\!\sqrt{\frac{1\!+\!2|x'|^2}{1\!+\!|x'|^2}}.\ee
We introduce the  ``back-projection" operator $P$ that sends  functions on $\partial B \times  (1, \infty)$ to  functions on $B$ by the formula
 \be\label{pfnh}
(PF)(x)=\frac{1}{|\partial B |}\,\intl_{ \partial B } F(\xi, [\xi, x])\,
d  \sig(\xi), \qquad x \in B.\ee
   Let 
 $$
 \tilde B=\{x'\in \bbr^n: |x'|< \sinh R\}$$
be the orthogonal projection of $B$ onto 
the  hyperplane $ x_{n+1}=0$. If $\xi=e_{n+1}\,\cosh R +\om\, \sinh R$,  $\;\om\in S^{n-1}$, and  $x=(x',x_{n+1})\in B$, then
 $$
  [\xi, x]= \sqrt{1+|x'|^2} \, \cosh R -(x'\cdot\om)\, \sinh R,$$ and $(PF)(x)$
  is actually a function of $x'\in \tilde B$.  We denote this function by 
 $(\tilde PF)(x')$.

\subsection{The case $n>2$}  Let, as above,    $\xi \in \partial B, \; t>1$. Consider the  analytic family of operators
\be
(N^\a f)(\xi, t)=\intl_{B} \frac{|[\xi, y] -t|^{\a -1}}{\Gam (\a/2)}\, f(y)\, dy, \qquad Re \, \a >0, \ee

\begin{lemma}\label{aauh}  If  $f\in C^\infty (\bbh^n)$, $\supp f \subset  B$, then
\be\label {anah} \underset
{\a=3-n}{a.c.} (\tilde PN^\a
f)(x')=\frac{\Gam (n/2)\, (\sinh R)^{2-n}}{\pi^{1/2}} \intl_{ \tilde B} \frac {\tilde f(y')}{|x'-y'|^{n-2}}\,dy',\ee
\be
\tilde f(y')=f(y',  \sqrt{1+|y'|^2})\,\sqrt{\frac{1+2|y'|^2}{1+|y'|^2}}.
\ee
\end{lemma}
\begin{proof} For $Re \, \a >0$, changing the order of
integration, we obtain
$$
(PN^\a f)(x)\!=\!\intl_{B} \!f(y)\, k_\a (x,y) \, dy, \quad k_\a (x,y)\!=\!\frac{1}{|\partial B |}\,\intl_{\partial B} 
\frac{|[\xi, (x\!- \!y)]|^{\a -1}}{\Gam (\a/2)}
d  \sig(\xi).$$
Since $\xi$ has the form $\xi=e_{n+1}\,\cosh R +\om\, \sinh R$,  $\om\in S^{n-1}$, then
\bea
|[\xi, (x- y)]|&=&|(x_{n+1}-y_{n+1})\, \cosh\,R-(x'- y')\cdot \om\, \,  \sinh R|\nonumber\\
&=& \left | h-\om \cdot  \sig\right |\, |x'- y'|\, \sinh R,\nonumber\eea
\be h=\frac{x_{n+1}-y_{n+1}}{|x'- y'|}\, \coth R, \qquad \sig=\frac{x'- y'}{|x'- y'|}.\ee
Hence,
\bea
k_\a (x,y)&=& \frac{(|x'- y'|\, \sinh R)^{\a -1}}{\sig_{n-1} } \intl_{S^{n-1}} \frac{|h-\om \cdot \sig |^{\a -1}}{\Gam (\a/2)}\, d\om\nonumber\\
&=& \frac{\sig_{n-2}\,(|x'- y'|\, \sinh R)^{\a -1}}{\sig_{n-1}} \, g_\a (h), \quad \mbox{\rm  cf. (\ref{jui}).} \nonumber\eea
If $|h|< 1$, we can
apply Lemma \ref{lem3}   and write the  integral over $B$ as that over $\tilde B\subset \bbr^n$. This will give the result.

It remains to show that $|h|< 1$.  By symmetry we may suppose that $\left\vert y^{\prime
}\right\vert \leq\left\vert x^{\prime}\right\vert ,$ which we shall do from
now on. Let 
$$
a=|y'|,\qquad b=|x'|, \qquad b_0=\sinh R.$$
Since $|x'-y'| \ge |x'|-|y'|$, then
$$
h\le f_a\left(  b\right)\,\coth R, \qquad  f_a\left(  b\right)  =\frac{\sqrt{1+b^{2}}-\sqrt{1+a^{2}}}{b-a}.$$
For $a$ fixed, the function  $ f_a\left(  b\right)$ is increasing in $(a,\infty),$ because
\[f_a^{\prime}\left(  b\right)  =\frac{D\left(  a,b\right)  }{\left(  b-a\right)
^{2}\sqrt{1+b^{2}}},     \quad D\left(  a,b\right)  =\sqrt{\left(  1+b^{2}\right)  \left(  1+a^{2}\right)
}-1-ab>0.\]
Hence, $h\le f_a (\sinh R)\,\coth R$. The right-hand side of this inequality is less than $1$.  Indeed, setting $b_0=\sinh R$, we have

\begin{align*}
&f_a (\sinh R)\,\coth R=
 \frac{\sqrt{1+b_0^{2}}-\sqrt{1+a^{2}}}{b_0-a}\, \frac{\sqrt{1+b_0^{2}}}{b_0}<1,\\
& \Leftrightarrow\frac{1+b_0^{2}-\sqrt{\left(  1+a^{2}\right)  \left(
1+b_0^{2}\right)  }}{\left(  b_0-a\right)  b_0}<1,\\
& \Leftrightarrow1+b_0^{2}-\sqrt{\left(  1+a^{2}\right)  \left(  1+b_0^{2}\right)
}<\left(  b_0-a\right) b_0,\\
& \Leftrightarrow0<D\left(  a,b_0\right)  \,.
\end{align*}
This completes the proof.
\end{proof}

\begin{lemma}\label{aaush}    
Let $f\in C^\infty (\bbh^n)$, $\; \supp f \subset B$.  Then
\be\label {anxh}  \underset
{\a=3-n}{a.c.} (PN^\a
f)(x)=\frac{\del_{n}}{(\sinh R)^{n-1}}\,  \intl_{\partial B} (d/dt)^{n-3} [(Mf)(\xi, t)\, (t^2-1)^{n/2 -1}]\Big |_{t=[\xi, x]}\, d\xi,\nonumber\ee
if  $n=3,5, \ldots\, $, and 
\bea\label {anx2h} &&\underset
{\a=3-n}{a.c.} (PN^\a
f)(x)=-\frac{\del_{n}}{\pi\,(\sinh R)^{n-1}} \,  \intl_{\partial B}  
d\xi\nonumber\\
&&\times\intl_1^{\cosh\, 2R}  (d/dt)^{n-2} [(Mf)(\xi, t)\, (t^2-1)^{n/2 -1}]\,\log |t-[\xi, x]|\,dt,\nonumber\eea
 if $n=4,6,\ldots\,$, where $\del_{n}$ is defined by (\ref{amyw}).
\end{lemma}

\begin{proof} For $Re \, \a >0$, by making use of the formula
\be\label{8gh}
\intl_{\bbh^n} f(y)\, a ([\xi, y]) \, dy=\sig_{n-1}\intl_{1}^\infty  a(\t) (Mf)(\xi, \t)\,(\t^2-1)^{n/2 -1}\, d\t,\ee
we have
\bea
&&(N^\a f)(\xi, t)=\frac{\sig_{n-1}}{\Gam (\a/2)}\intl_{1}^\infty  (Mf)(\xi, \t)\,|\t -t|^{\a -1} (\t^2-1)^{n/2 -1}\, d\t\nonumber\\
&&= \intl_\bbr \frac{|\t|^{\a -1}}{\Gam ( \a/2)} \,  \vp_\xi (\t+t)\, d\t, \quad  \vp_\xi (\t)=\sig_{n-1}\,(Mf)(\xi, \t) \,(\t^2-1)^{n/2 -1}_+.\nonumber\eea
Since $f$ is smooth and the  support of $f$ is separated from the boundary 
 $\partial B$, then $(Mf)(\xi, \t)$ is smooth  in the $\t$-variable uniformly in $\xi$  and vanishes identically in the respective neighborhood of $\t= 1$. Thus,  
 Lemma \ref{lem1} yields the following equalities:

\noindent For $n=3,5, \ldots\, $:
$$
 \underset
{\a=3-n} {a.c.}\,(N^\a f)(\xi, t)\!=\!\del_{n} \,  \vp_\xi^{(n-3)} (t),$$

\noindent For $n=4,6,\ldots\,$:
$$
 \underset
{\a=3-n}{a.c.}\,(N^\a f)(\xi, t)=-\frac{\del_{n}} {\pi}\,\intl_1^{\cosh\, 2R}\!\!\vp_\xi^{(n-2)} (\t)\,\log |\t\!-\!t|\, d\t,$$
$\del_{n}$ being the constant from  (\ref{amyw}). Now the result follows; 
cf.  Lemmas \ref{aau2}  and \ref {aauss}.
\end{proof}

 Lemmas   \ref{aauh}  and   \ref{aaush} imply the following inversion result for the spherical means on $\bbh^n$. We recall that  $\Del_{x'}$ denotes the usual Laplace operator in the $x'$-variable.

\begin{theorem} \label{774h}  Let $n>2$. An infinitely differentiable function  $f$ supported
in the geodesic ball $B=\{x\in  \bbh^n: \dist (x,e_{n+1})<R\}$,  can be reconstructed from its spherical means 
$\left(Mf\right)  \left(\xi,\t\right)$, $(\xi, t)\in \partial B \times (1,\infty)$,  by the formula
\be\label{8nj} f(x)\!=\!  \frac{d_n\, x_{n+1}}{|x|\,\sinh R}\, \Del_{x'} f_0 (x', \sqrt{1\!-\!|x'|^2}),\quad d_{n}\!=\!\frac{(-1)^{[n/2-1]}}{2^{n-1} \pi^{n/2 -1}\Gam(n/2)},\nonumber\ee
where  $|x|=   \sqrt{|x'|^2+x^2_{n+1}} $    and $f_0 (x)\equiv f_0 (x', \sqrt{1+|x'|^2})$ has the following form:
$$
f_0 (x)=-\intl_{\partial B} (d/dt)^{n-3} [(Mf)(\xi, t)\, (t^2-1)^{n/2 -1}]\Big |_{t=[\xi, x]}\, d\xi,
$$
if $n=3,5, \ldots\, $, and
$$
f_0 (x)=\frac{1}{\pi} \intl_{\partial B} d\xi
\intl_1^{\cosh\, 2R}  (d/dt)^{n-2} [(Mf)(\xi, t)\, (t^2-1)^{n/2 -1}]\,\log |t-[\xi, x]|\,dt$$
if $n=4,6,\ldots\,$.
\end{theorem}

\subsection{The case $n=2$}  The argument  follows Section \ref{5477} almost verbatim. 
 Let
\be\label {87fh}
(I_{\ast}f)(x)=\frac{1}{2\pi}\intl_{B} f(y)\log|x'-y'|\,dy,\ee
so that 
\be\label{lazh} \Delta_{x'} (I_{\ast}f)(x)=\tilde f (x')=|x|\, f (x)/x_3. \ee

\begin{lemma}\label{nn7h} If $f$ be a $C^\infty$  function supported
in $B$, then
\be\label{kukh}
(I_{\ast}f)(x)\!=\!\frac{1}{|\partial  B|}\intl_{\partial B
}\int\limits_{1}^{\infty}\!\left(Mf\right)  \left(\xi,\t\right)  \log |\t-
[\xi, x] |\,d\t\,d\xi\!+\!c_f,
\ee
$$
c_f=-\frac{1}{2\pi } \, \Big (\log \, \frac{\sinh R}{2} \Big)\,\intl_{B} f(y)\, dy.
$$
\end{lemma}
\begin{proof}  Let
$$
\left( N_{\ast}f\right)(\xi, t) =\intl_{B} f(y)\log|[\xi, y]-t|\,dy, \qquad (\xi, t)\in\partial B \times  (1,\infty).
$$
Changing the order of
integration,   we obtain
$$ (PN_{\ast} f)(x)=\intl_{B} f(y)\,k_{\ast} (x,y)\, dy,
$$ 
 $k_{\ast} (x,y= \log|x'-y'|+\log \, \sinh R -\log2$ (see the proof of Lemma \ref{nn7s})
 This gives
\be\label{90kh}
(PN_{\ast} f)(x)=2\pi\,(I_{\ast}f)(x)+\Big (\log \, \frac{\sinh  R}{2} \Big)\,\intl_{B} f(y)\, dy.\ee
On the other hand, by (\ref{8gh}), 
\be\label{m23h}
(PN_{\ast} f)(x)=\frac{1}{ \sinh  R} \,\intl_{\partial B} \intl_{1}^\infty (Mf)(\xi, \t) \log |\t -[\xi, x]|\, d\t d\xi.
\ee
Comparing (\ref{m23h})  with (\ref{90kh}),  we obtain (\ref{kukh}).
\end{proof}

Owing to (\ref{lazh}),
Lemma \ref{nn7h}  allows us complete  Theorem \ref{774h} as follows.$\smallskip$

\begin{theorem}
\label{Theorem n2.1h}  An infinitely differentiable function  $f$ supported
in the geodesic ball $B=\{x\in  \bbh^2: \dist (x,e_{3})<R\}$ can be reconstructed from its spherical means 
$\left(Mf\right)  \left(\xi,\t\right)$, $(\xi, t)\in \partial B \times (1,\infty)$,  by the formula

\begin{equation}\label {finh}
f\left(x\right)  =\frac{x_3}{2\pi\,  |x| \,\sinh R}\,\Delta_{x'}  \intl_{\partial B
}\int\limits_{1}^{\infty}\!\left(Mf\right)  \left(\xi,\t\right)  \log |\t-
[\xi, x ]|\,d\t\,d\xi.
\end{equation}
\end{theorem}

\begin {remark}  As in Section
 \ref{45er}, Theorems \ref{774h}  and \ref{Theorem n2.1h}  can be applied to solution of inverse problems for the EPD equation in the hyperbolic space. The reasoning follows the same lines as before.  We leave it to the interested reader.
\end {remark}

\section {Appendix: Proof of Lemma \ref   {lem3}} \label {appx}
It is convenient to split the proof in two parts.

\noindent (i) We recall the notation
\be\label {denz}
g_\a (h)=\frac{1}{\Gam (\a/2)} \intl_{-1}^1 |t \!-\!h|^{\a -1}\,(1-t^2)^{(n-3)/2}\, dt , \qquad  Re\, \a >0,\ee
where $n>2$ and $ |h|\le 1-\del,\; \del>0$.  Changing variables $t=2\t -1$, $h=2\xi -1$, we write 
$g_\a (h)\equiv G_\a ((1+h)/2)$, where 
 \bea
G_\a (\xi)&=&\frac{2^{\a+n-3}} {\Gam(\a/2)}\intl_{0}^1  |\t \!-\!\xi|^{\a -1}\,(1-\t)^{(n-3)/2}\,\t^{(n-3)/2}\,  dt\nonumber\\
 &=&U_\a (\xi)+U_\a (1-\xi), \qquad \del/2\le\xi \le 1-\del/2,\label{908b}\eea
 $$
 U_\a (\xi)=\frac{2^{\a+n-3}} {\Gam(\a/2)}\intl_0^\xi  (\xi -\t)^{\a -1} \t^{(n-3)/2} (1-\t)^{(n-3)/2}\, d\t.$$
 The last integral expresses through the Gauss hypergeometric function so that
 $ U_\a (\xi)=a_\xi (\a)\, b(\a) \, \z_\a (\xi)$, where
  $$a_\xi (\a)=2^{\a+n-3}\,\xi^{(n-3)/2 +\a}\, \Gam ((n-1)/2) , \qquad  b(\a)=\frac {\Gam (\a)}{\Gam (\a/2)},$$ 
 $$
 \z_\a (\xi)=\frac{1}{\Gam (\a+(n-1)/2)}\, F\left (\frac{n-1}{2}, \frac{3-n}{2}; \frac{n-1}{2}+ \a; \,\xi\right ); $$
 see, e.g., \cite[2.2.6(1)]{PBM}.   Owing to \cite [2.1.6]{Er}, 
  $\z_\a (\xi)$ extends as an entire function of $\a$,  which is represented by an  absolutely convergent power series. Since
   $\del/2\le\xi \le 1-\del/2$, this series converges  uniformly in $\a  \in  K$ for any compact subset $K$ of the complex plane. Furthermore,  $a_\xi (\a)$ is also an entire function and  $b(\a)$  is meromorphic with the only poles $-1, -3, -5, \ldots$. Since $g_\a$ is an even function, i.e.,  $g_\a (h)=g_\a (-h)$, these poles are eventually removable. 
   Hence, $G_\a (\xi)$ extends to all complex $\a$  as an entire function of $\a$ and this extension represents a $C^\infty$ function of $\xi$  uniformly in $\a\in K$. This gives the desired result for $g_\a (h)$.
 
\noindent (ii)  To compute analytic continuation of $g_\a$ at $\a=3-n$,  first, we assume  $1/2<Re\,\a<1$ 
and $|Im \,\a|<1$ 
 and represent $G_\a (\xi)$  as a Mellin convolution
\be\label{conv}
 G_\a (\xi)\!=\frac{2^{\a+n-3}}{\Gam(\a/2)} \, f(\xi), \quad
 f(\xi)=\intl_{0}^\infty  
f_1(\tau)f_2\left(\frac{\xi}{\tau}\right)\frac{d\tau}{\tau}, 
\ee
 where
\[
f_1(\tau)\!=\! \left\{ \!
 \begin{array} {ll} \tau^{\a+(n-3)/2}(1\!- \!\tau)^{(n-3)/2} & \mbox{if $0\!<\!\tau\!<\!1$,}\\
 {}\\
0,& \mbox{if $1\!<\!\tau\!<\!\infty$,}\\
 \end{array}\
\right.\qquad  f_2(\tau)=|1\!-\!\tau|^{\a-1}.
 \]
  The   Mellin transforms $\tilde f_j(s)\!=\!\int_{0}^\infty  \!f_j(\tau)\tau^{s-1}d\tau$  ($j\!=\!1,2$) are   
  evaluated as  
    $$
 \tilde f_1(s)=\frac{\Gam(s+\a+(n-3)/2)\Gam((n-1)/2)}{\Gam(s+\a+n-2)},\qquad Re\, s>\frac{3-n}{2}-Re\, \a\, ,
$$
$$
\tilde f_2(s)=\frac{\Gam(s)\Gam(\a)}{\Gam(s+\a)}+\frac{\Gam(1-s-\a)\Gam(\a)}{\Gam(1-s)}, \qquad 0< Re\, s<1-Re\, \a.
$$
By applying the convolution theorem  and the relevant Mellin inversion formula \cite{T}, we obtain
 $$
f(\xi)=\frac{1}{2\pi i} \intl_{\kappa -i\infty}^{\kappa +i\infty}  
\tilde f(s)\,  \xi^{-s}\,ds,  \qquad 0 <\kappa<1-Re\,\a,
$$ 
where $\tilde f (s)=\tilde f_1(s)\tilde f_2(s)$.  The function $\tilde f (s)$  has  poles in the half-plane $Re\, s<\kappa$
at the points $s=-j$ and $s=-j-\a-(n-3)/2$, $j=0,1,2,\ldots$.  Since $1/2<Re\,\a<1$,  all these poles are simple,  and 
 the Cauchy residue theorem yields
 \bea 
&&\! \! \! \! \! \! \! f(\xi)=\Gamma\left(\frac{n-1}{2}\right) \, \Gamma(\a) \sum_{j=0}^\infty
\frac{(-\xi)^j}{j!}\left [\frac{\Gamma(\a-j+(n-3)/2)}
{\Gamma(\a-j)\Gamma(\a-j+n-2)}\right .\nonumber\\
&&\! \! \! \!\! \!  \!+\left .\frac{ \xi^{\a+(n-3)/2}}
{\Gam((n-1)/2-j)}\left (\frac{\Gam(-\a+(3-n)/2-j)}{\Gam((3-n)/2-j)}
+\frac{\Gam((n-1)/2+j)}
{\Gam(\a+(n-1)/2+j)}\right )\right].\nonumber\eea
 Ultimately,
we arrive at the following expression for  $G_\a (\xi)$: 
 \bea\label {Gauss}
G_\a (\xi)&=&\lam_1 \,F\left(1-\a,3-\a-n;\frac{5-n}{2}-\a;\xi\right)\\&+&\lam_2 \,
 F\left(\frac{3-n}{2},\frac{n-1}{2};\frac{n-1}{2}+\a;\xi\right),\nonumber\eea
 where
\bea
\lam_1 &=&\frac{\Gam((n-1)/2)}{2^{3-\a-n}\Gam(\a/2)} 
\, \frac{(-1)^n\,\Gam(3-\a-n)\, \sin\a\pi}{\Gam((5-n)/2-\a)\,\cos\,(\a+n/2)\pi},\nonumber\\
\lam_2 &=&\frac{\Gam((n-1)/2)}{2^{3-\a-n}\Gam(\a/2)}\, 
 \frac{\xi^{\a+(n-3)/2}\,\Gam(\a)}{\Gam(\a+(n-1)/2)}\, \left(1+\frac{\cos \,n\pi/2}{\cos\,(\a+n/2)\pi}\right).\nonumber\eea 

Case 1. Let $n=2m$, $m=2,3,\ldots$. Then
\be\label {even1}
 G_\a (\xi)\!=\frac{\pi\,\Gam(m-1/2)}{2^{3-\a-2m}\, \Gam(\a/2)\,\cos\,\a\pi}\,[D_1(\xi;\a)+D_2(\xi;\a)],
 \ee
 where
 \bea
 D_1(\xi;\a)&=&
 \frac{ F(1-\a,3-\a-2m;5/2-\a-m;\xi)}{(-1)^m\Gam(\a+2m-2)\,\Gam(5/2-\a-m)},
 \nonumber\\
 D_2(\xi;\a)&=&
\frac{\cot (\a\pi/2)\,F(3/2-m,m-1/2,\a+m-1/2;\xi)}{\xi^{3/2-\a-m}\,\Gam(1-\a)\,\Gam(\a+m-1/2)}.\nonumber\eea 
A simple computation yields
$
\underset
{\a=3-2m}{a.c.}\, G_\a (\xi)\!=\!\Gam(m\!-\!1/2)=\Gam((n\!-\!1)/2).
$

Case 2. Let  $n=2m+1$, $m=1,2,\ldots$. Then
\be\label {odd1}
G_\a (\xi)\!=\frac{\Gam(m)\,\Gam(1-\a/2)}
{2^{2-\a-2m}\,\cos \, (\a\pi/2)}\, [E_1(\xi;\a)+E_2(\xi;\a)],
\ee
 where
\bea
 E_1(\xi;\a)&=&\frac{
F(1-\a,2-\a-2m;2-\a-m;\xi)}{(-1)^{m+1}\,\Gam(\a-1+2m)\,\Gam(2-\a-m)},
  \nonumber\\
  E_2(\xi;\a)&=&\frac{\xi^{\a+m-1}F(1-m,m;\a+m;\xi)}{\Gam(1-\a)\,\Gam(\a+m)}.
 \nonumber\eea 
Passing to the limit as $\a\to 2-2m$, we obtain $\underset
{\a=2-2m}{a.c.}\, G_\a (\xi)=\Gam (m)=\Gam((n\!-\!1)/2)$. This completes the proof. ${} \qquad \qquad \hfill \square$
 
 \begin{remark} The basic equality  (\ref {Gauss}) can be proved in a different way if we  rearrange  hypergeometric functions in (\ref{908b}) using known formulas.
  Specifically, the second term in (\ref{908b})  can be transformed by formulas (33), (6),  and (21) from \cite [Section 2.9]{Er}. This gives 
 $$ U_\a (1-\xi)= \frac{2^{\a+n-3}} {\Gam(\a/2)}\,B\left (\frac{n-1}{2}, \a\right )\, (A(\xi) +B(\xi)),$$
 \bea
 A(\xi)&=&\gam_1 (\a)\, F\left (1-\a, 3-n-\a; \frac{5-n}{2}- \a; \,\xi\right ),\nonumber\\
  B(\xi)&=& (\xi (1-\xi))^{\a+(n-3)/2} \gam_2 (\a)\, F\left (\a, \a+n-2; \frac{n-1}{2}+ \a; \,\xi\right ),\nonumber\eea
\bea
 \gam_1 (\a)&=&\frac{\Gam (\a+(n-1)/2)\,\Gam (\a+(n-3)/2)}{\Gam (\a)\, \Gam (\a+n-2)}, \nonumber\\
 \gam_2 (\a)&=&\frac{\Gam (\a+(n-1)/2)\,\Gam ((3-n)/2-\a)}{\Gam ((n-1)/2)\, \Gam ((3-n)/2)}=
 \frac{\sin (n-1)/2) \pi}{\sin (n-1)/2+\a) \pi} 
 .\nonumber\eea
  Owing to  \cite[2.9(2)]  {Er},
  $$
  B(\xi)=\gam_2 (\a) \, \xi^{(n-3)/2 +\a}\, F\left (\frac{n-1}{2}, \frac{3-n}{2}; \frac{n-1}{2}+ \a; \,\xi\right ).$$
  This gives (\ref{Gauss}).
  \end{remark}
  
 {\bf An alternative proof of (ii)}  The following alternative proof is  instructive and leads to the same result. 
  In fact,  it suffices to prove the equality
   \be\label{9dr}
\underset
{\a=3-n}{a.c.}\, g_\a (h)= \Gam ((n-1)/2)\ee
in the weak sense. Indeed,   suppose that  
 \be\label {aw2}
  \underset
{\a=3-n}{a.c.}\, (g_\a,\psi)\equiv \underset
{\a=3-n}{a.c.}\intl_\bbr g_\a (h) \,\psi (h)\, dh= \Gam ((n-1)/2) \intl_\bbr \psi (h)\, dh\ee
 for any $C^\infty$ function $\psi$ with compact support in the interval $(-1,1)$.
 Since, by 
  Part (i),  the analytic continuation of $g_\a (h)$  represents a $C^\infty$ function of $h$ uniformly in $\a\in K$ for any compact subset $K$ of the complex plane, then (use, e.g., \cite[Lemma 1.17 ]{Ru96})  $$\underset
{\a=3-n}{a.c.}\, (g_\a,\psi)=(\underset
{\a=3-n}{a.c.}\, g_\a,\psi)$$ and  (\ref{aw2}) yields $(\underset
{\a=3-n}{a.c.}\, g_\a,\psi)=\Gam ((n-1)/2)\, (1, \psi)$. This implies (\ref{9dr}). 

Let us prove (\ref{aw2}).  We denote
$$
\rho_\a (t)=\frac{|t|^{\a -1}}{\Gam ( \a/2)} , \qquad  \om (t)=(1-t^2)_+^{(n-3)/2},$$
where $(\cdot)_+$  stands for zero when the expression in brackets is non-positive.  We interpret these functions as $\D'$-distributions on $\bbr$. Then $g_\a$ is a convolution of $ \rho_\a$ with the compactly supported distribution $\om$   so that
$(g_\a,\psi)=(\rho_\a (s), (\om (t), \psi (s+t)))$; see, \cite[Ch. I, Sec. 4(2)]{GSh1}. 
 
 If $n$ is odd,  $n=2m+3$, $m=0,1, \ldots$, then (\ref {lab1c}) yields
\bea
 \underset
{\a=3-n}{a.c.}\, (g_\a,\psi)&=&\underset
{\a=-2m}{a.c.} (g_\a,\psi)=c_{m,1}\,\left (\frac{d}{ds}\right )^{2m} \, (\om (t),  \psi (s+t)) \Big |_{s=0}\nonumber\\
&=&c_{m,1}\,(\om,  \psi^{(2m)})= c_{m,1}\,([(1-t^2)^{m}]^{(2m)},  \psi(t))\nonumber\\
&=&(-1)^m c_{m,1}\, (2m)!=m!=\Gam ((n-1)/2).
\nonumber\eea

 Let now $n$ be  even. Since the convolution is commutative,
  $${a.c.}\, (g_\a,\psi)=a.c.\, (\om (t), (\rho_\a (s), \psi (s+t)))= (\om (t), a.c.\, (\rho_\a (s), \psi (s+t))).$$
 If $n=2m+2$ ($m=1,2, \ldots$), then, applying  (\ref{lab1z})  and changing variables, we have
 \bea
&& \underset
{\a=3-n}{a.c.}\, (g_\a,\psi)=\underset
{\a=1-2m}{a.c.} (g_\a,\psi)=c_{m,2}\,   \Big ( \om (t), \mbox {\rm p.v.} \!\intl_\bbr \frac{\psi^{(2m-1)} (s+t)}{s}\, ds
\Big )\nonumber\\
&=&(-1)^{m+1} c_{m,2} \intl_{-1}^1 \psi^{(2m-1)} (h)\, q(h) \,dh=(-1)^{m} c_{m,2}\intl_{-1}^1 \psi (h)\, q^{(2m-1)} (h) \,dh,\nonumber\eea
$$
c_{m,2}=\frac{1}{\Gam (1/2 -m) \,(2m-1)!}, \qquad q(h)=\mbox {\rm p.v.} \! \intl_{-1}^1 \frac{ (t^2\!-\!1)^{m}}
{(t-h)\, \sqrt{1-t^2}}\,dt$$
(interchange of the order of integration  can be  justified if we complete $[-1,1]$ to a closed contour and use  \cite[Section 7.1]{Gak}). Now, $ q(h)$ is a polynomial with leading term $ \pi h^{2m-1}$. This follows from the well known relation for Chebyshev polynomials
\be\label {pvn}
\mbox{\rm p.v.}\intl_{-1}^1  \frac{T_n(t)\, dt}{(t-h)\, \sqrt{1-t^2}}=\pi U_{n-1}(h), \quad -1<h<1,
\ee
and the fact that the leading terms of $T_n(h)$ and $U_n(h)$
are $2^{n-1}h^n$ and $2^n h^n$, respectively;  see formulas 10.11(47), 10.11(22),  and 10.11(23) in \cite[vol. II)]{Er}.
Hence integration by parts yields
$$
 \underset
{\a=3-n}{a.c.}\, (g_\a,\psi)=(-1)^{m} c_{m,2}\,\pi \, (2m-1)! \intl_{-1}^1  \psi (h)\,dh=\Gam (m+1/2) \intl_{-1}^1  \psi (h)\,dh,$$
where $\Gam (m+1/2)=\Gam ((n-1)/2)$. Thus, we are done. $\hfill \square$

\bibliographystyle{amsplain}

\end{document}